\documentclass[11pt]{amsart}
\usepackage{setspace}
\usepackage{stmaryrd}
\usepackage{amsmath, amsfonts, amsthm, amstext, amssymb, xspace}
\usepackage{eucal}
\usepackage{multirow}
\usepackage{amssymb,latexsym}
\usepackage[dvipsnames]{xcolor}
\usepackage{float}
\usepackage[shortlabels]{enumitem}
\pdfoutput=1 
\usepackage{microtype}
\usepackage{longtable,booktabs}
\usepackage{tikz-cd}
\usetikzlibrary{cd, shapes.geometric, calc}
\usepackage{upgreek}
\usepackage[all]{xy}
\usepackage{soul}
\usepackage{multicol}
\usepackage{hyperref}
\usepackage[capitalise, nameinlink]{cleveref}
\hypersetup{
    colorlinks,
    linkcolor={SecBlue},
    citecolor={PeoBlue},
    urlcolor={blue!80!black}
}
\usepackage[parfill]{parskip}
\usepackage{amsthm}

\theoremstyle{definition}
\newtheorem{thm}{Theorem}[section]
\newtheorem{lemma}[thm]{Lemma}

\newtheorem{proposition}[thm]{Proposition}

\newtheorem{notn}[thm]{Notation}
\newtheorem{conjecture}[thm]{Conjecture}
\newtheorem{remark}[thm]{Remark}
\newtheorem{definition}[thm]{Definition}

\newtheorem{technique}{Technique}

\newtheorem{convention}[thm]{Convention}
\newtheorem*{corollary*}{Corollary}
\newtheorem{application}{Application}
\newtheorem*{linearity}{$\kappabar$-linearity}

\newtheorem{thmx}{Theorem}

\def\m{{\sf m}}
\def\y{{\sf y}}

\def\FF{\mathbb{F}}

\tikzset{%
  symbol/.style={
    draw=none,
    every to/.append style={
      edge node={node [sloped, allow upside down, auto=false]{$#1$}}
    },
  },
}

\newcommand{\kappabar}{\overline{\kappa}}

\def\colim{\operatorname{colim}}
\def\ker{\operatorname{ker}}
\def\coker{\operatorname{coker}}

\def\im{\operatorname{Im}}

\def\tmf{{\sf tmf}}
\newcommand{\Y}{\mathrm{Y}}
\newcommand{\M}{\mathrm{M}}
\newcommand{\A}{\mathrm{A}}

\newcommand{\xra}{\xrightarrow}

\newcommand{\sma}{\wedge}

\def\upperalphabet{\\A\\B\\C\\D\\E\\F\\G\\H\\I\\J\\K\\L\\M\\N\\O\\P\\Q\\R\\S\\T\\U\\V\\W\\X\\Y\\Z}
\def\loweralphabet{\\a\\b\\c\\d\\e\\f\\g\\h\\i\\j\\k\\l\\m\\n\\o\\p\\q\\r\\s\\t\\u\\v\\w\\x\\y\\z}

\newcommand{\mr}[1]{\mathrm{#1}}

\def\makecommands#1#2#3{
    \bgroup
    \def\tempcmdname##1{#1}
    \def\tempcmdbody##1{#2}
    \def\\##1{\expandafter\xdef\csname\tempcmdname{##1}\endcsname{\unexpanded\expandafter{\tempcmdbody{##1}}}}
    #3
    \egroup
}
\makecommands{#1mr}{\mathrm{#1}}{\upperalphabet}
\makecommands{#1cal}{\mathcal{#1}}{\upperalphabet}
\makecommands{#1#1}{\mathbb{#1}}{\upperalphabet}
\makecommands{#1bar}{\overline{#1}}{\upperalphabet}
\makecommands{#1twee}{\widetilde{#1}}{\upperalphabet\loweralphabet}
\makecommands{sf#1}{{\sf #1}}{\upperalphabet\loweralphabet}
 
\definecolor{DarkBlue}{rgb}{.1, 0.35, 0.6} 
\definecolor{PeoBlue}{rgb}{.15, 0.8, 0.6} 
\definecolor{SecBlue}{rgb}{.15, 0.4, 0.4} 
\definecolor{DarkBrown}{rgb}{.5, 0.2, 0.2}

\newcommand{\irina}[1]{\marginpar{\color{VioletRed}#1}}
\newcommand{\prasit}[1]{\marginpar{\color{DarkBlue}#1}}
\newcommand{\pcom}[1]{{\color{DarkBlue}#1}}

\author{Prasit Bhattacharya}\address{New Mexico State University}\email{prasit@nmsu.edu}
\author{Irina Bobkova}\address{Texas A\&M University}\email{ibobkova@tamu.edu}
\author{J.D. Quigley}\address{University of Virginia}\email{mbp6pj@virginia.edu}

\title{New infinite families in the stable homotopy groups of spheres}

\begin{document}

\begin{abstract}
We  identify seven new  $192$-periodic  infinite families of elements in the  $2$-primary stable homotopy groups of spheres, all of which have trivial image under the Hurewicz map for topological modular forms.   We prove that these families are nontrivial  after  $\mathrm{T}(2)$- as well as  $\mathrm{K}(2)$-localization. We also obtain new information about $2$-torsion and $2$-divisibility of some of the previously known $192$-periodic infinite families in the stable stems. 
\end{abstract}

\maketitle


\section{Introduction}

Calculating the stable homotopy groups of spheres has been one of the central problems in algebraic topology  for several decades, with many applications in algebra and  geometry. Since the 1960s, there have been two main approaches: low-dimensional computations,  which  describe the stable stems in a finite range using Adams spectral sequences \cite{MT67, Rav86, KM93,  Isa19, IWX23}, and chromatic computations, which instead aim to identify   large-scale periodic patterns  \cite{  Ada66, Smi77, MRW77, Rav86}.

The first large-scale phenomenon observed  in the stable stems is Serre's finiteness theorem that all positive dimensional stable stems are finite abelian groups \cite{Ser53}. This motivated the study of the  stable stems one prime at a time. In the 1960s,  Toda \cite{Tod62}  identified several $(2p-2)$-periodic families of $p$-torsion elements for primes $p\geq 3$, and Adams \cite{Ada66} identified several $8$-periodic families when $p=2$. A decade later, Smith \cite{Smi77} constructed  $(2p^2-2)$-periodic families for $p \geq 5$, and Miller--Ravenel--Wilson \cite{MRW77} constructed  $(2p^3-2)$-periodic families for $p\geq 7$. These are the first examples of  \emph{periodic phenomena}, which motivated the development of chromatic homotopy theory in the subsequent decades. 


Chromatic homotopy theory implies that  the $p$-local stable stems admit  a decreasing filtration indexed by the natural numbers, where the  $n$-th stratum is called the \emph{chromatic layer $n$}.  The  fundamental results on  nilpotence and  periodicity  \cite{DHS88, HS98}, together with the chromatic convergence theorem \cite{Orange}, imply that chromatic layer $n$ consists entirely of infinite periodic  families. These families are detected by  periodic $v_n$-self-maps of  \emph{type $n$} CW complexes. Adams  completely identified  chromatic layer $1$ \cite{Ada66} in the 1960s, but half a century later, chromatic layer $2$ is still an active area of research \cite{ShimomuraWang, SY, GHMR, GHMV1,  Beh, Behrens5}.  

 Periodic $v_{n}$-self-maps are defined as the self-maps of a $p$-local type $n$ complex that are detected by the $n$-th Morava K-theory at the prime $p$. The coefficient ring for this theory is $\mr{K}(n)_*\cong \FF_p[v_n^{\pm 1}]$, where  $|v_{n}|=2(p^{n}-1)$. A $v_{n}$-self-map detected by the element $v_{n}^{k}$ (often referred to as a $v_{n}^{k}$-self-map) gives rise to $k|v_{n}|$-periodic families within the stable stems. The known calculation of stable stems  suggests that the period index $k$ cannot equal $1$ unless the prime $p$ is sufficiently large relative to the height $n$ (see the table below). These larger periods are among several issues that complicate computations in chromatic homotopy theory at small primes,  with the prime $p=2$ being the most challenging (see \cref{Table:periodicity}). 

\begin{table}[h]
\scriptsize
\renewcommand{\arraystretch}{1.5}
	\begin{tabular}{ c|p{2.8cm}|p{2.4cm}|p{2.4cm}|p{2.4cm} } 
		&$p=2$ & $p=3$ & $p=5$ &$p\geq 7$\\
		\hline
		$n=1$& $2^2|v_1|$ \cite{Ada66}& $|v_1|$ \cite{Ada66, Tod62}& $|v_1|$ \cite{Ada66, Tod62}& $|v_1|$ \cite{Ada66, Tod62}\\ 
		\hline
		$n=2$& $2^5|v_2|$ \cite{BHHM20}\newline $2^5|v_2|$ \cite{BMQ23} \newline $2^5|v_2|$ \Cref{MT:Families} & $3^2|v_2|$ \cite{BP04} & $|v_2|$ \cite{Smi77} &$|v_2|$ \cite{Smi77} \\ 
		\hline
		$n=3$ &? &? &? &$|v_3|$ \cite{MRW77} \\
		\hline	
		$n \geq 4$ &? &? &? & ?  \\	
	\end{tabular}
	\vspace{10pt} \\
	\caption{Known periodicity among infinite families at height $n$ and prime $p$.}\label{Table:periodicity} \vspace{-10pt}
	\end{table}
Investigating chromatic layer $2$ at $p=2$ has been the focus of much recent work \cite{Beau15, BO16, BG18, BHHM20, BMQ23, BBGHPS24, BGH}. One approach, the duality resolution \cite{BBH}, shows promise for identifying numerous new infinite families, but significant technical barriers hinder the complete analysis of the resolution when $p=2$. Another approach, the $\tmf$-based Adams spectral sequence, has yielded recent results; for example the work in \cite{BHHM20, BMQ23} identifies the zero line of this spectral sequence, resulting in new $192$-periodic infinite families within the Hurewicz image of $\tmf$. But like the duality resolution, it is difficult to push this method and detect elements beyond the Hurewicz image of $\tmf$. 

Our main result identifies seven new $2$-local 192-periodic families in the stable stems. These are the first examples of infinite families which are \emph{not} in the Hurewicz image of $\tmf$: 
\begin{thmx}\label{MT:Families} \emph{
For each $m \in \{23, 47, 71, 74, 95, 119, 167 \}$ and $k \in \mathbb{N}$, there exists  an element of  order $2$ in dimension $m+192k$ of the  stable stems whose image is trivial under the $\tmf$-Hurewicz homomorphism.   }
\end{thmx}
\begin{remark} \label{rmk:predict} A comparison of our work with known calculations  \cite{Isa19, IWX23}  suggests that  
the elements  ${\sf Ph}_1{\sf d}_0$, ${\sf h}_1^2 \cdot (\Delta {\sf h}_1{\sf g})$, ${\sf  h}_1^2 \cdot (\Delta^2 {\sf h}_1 {\sf g})$, ${\sf d}_0{\sf g}^3$, $(\Delta {\sf h}_1)^3{\sf g}$, $\Delta^4 {\sf h}_1^3 {\sf g}$, and $\Delta^6 {\sf h}_1^3 {\sf g}$  in the chart of \cite{IWX23} detect  the elements in dimension $23$, $47$, $71$, $74$, $95$, $119$, and  $167$ of \Cref{MT:Families}, respectively.\footnote{We thank Bob Bruner and Dan Isaksen for this observation.}
\end{remark}


The \(2\)-local connective spectrum of topological modular forms $\tmf$, a key construction in spectral algebraic geometry \cite{Ltmf, Btmf}, serves as a powerful tool for exploring chromatic height $2$ at the prime $2$. Its efficacy stems from the fact that $\tmf$-homology groups are readily computable, while its coefficient ring $\tmf_*$ exhibits intricate patterns reflecting the structure of the second chromatic layer \cite{Bau08, DFHH14, BR21}. 

Recent years have seen new techniques involving $\tmf$ \cite{BOSS19, BBT21, BBC23} yield significant results \cite{BHHM20, Bobkova20, BElocal20, BBBCX21} concerning the chromatic layer $2$. This progress culminated in the paper \cite{BMQ23}, which completely calculates the image of the $2$-primary Hurewicz homomorphism
\[ 
 \begin{tikzcd}
{\sf h}_{\tmf}: \pi_*\SS\rar & \tmf_*.
\end{tikzcd}
\]  A major consequence of this work is the identification of numerous previously unknown $192$-periodic infinite families and patterns within the second chromatic layer, all with nontrivial image in the $\mr{K}(2)$-local and $\mr{T}(2)$-local stable stems.

The identification of infinite families in \Cref{MT:Families} uses a $v_2$-self-map. This implies that they inherently possess a nonzero image within the $\mr{T}(2)$-local stable stems, a consequence of a general theory fleshed out in \Cref{sec:T2}. We further establish that these elements also have a nonzero image in the  $\Kmr(2)$-local stable stems in \Cref{sec:K2}:
 \begin{thmx}[{\cref{MT:T2} and \cref{MT:K2}}]\label{MT:KT2}
 \emph{ All elements listed in \Cref{MT:Families}  have nonzero images in the  $\mr{T}(2)$-local and $\mr{K}(2)$-local stable stems at $p=2$.}
  \end{thmx}

 
 
Beyond their significance in chromatic homotopy theory, the existence of these 192-periodic families has profound implications for geometric topology, specifically the classification of smooth structures on spheres. Recall that an exotic $n$-sphere is a smooth $n$-dimensional manifold which is homeomorphic, but not diffeomorphic, to $S^n$ with its standard smooth structure.  The foundational work of Kervaire and Milnor \cite{KM63}  relates the stable stems to the classification of smooth structures on homotopy spheres. The work of Kervaire and Milnor  \cite{KM63}, Browder \cite{Bro69}, Hill, Hopkins, and Ravenel \cite{HHR16}, and Wang and Xu \cite{WX17} implies that exotic spheres exist in every odd dimension except for $1$, $3$, $5$, and $61$. 

The even-dimensional case is more elusive. Classical results of Adams and Toda show that exotic spheres exist in at least one quarter of all even dimensions, and more recently, work of Behrens, Hill, Hopkins, Mahowald \cite{BHHM20} and Behrens, Mahowald, and Quigley \cite{BMQ23}  extended this to over half of all even dimensions. These developments support a perspective, prominently suggested by Wang and Xu \cite[Conjecture 1.17]{WX17}, that the existence of unique smooth structures is exceedingly rare. Specifically, it is expected that for $n > 4$, a unique smooth structure exists only in dimensions $5$, $6$, $12$, $56$, and $61$. The identification of new periodic families at height $2$ is a crucial step toward verifying this picture, as the work of Kervaire and Milnor implies that these families give rise to infinite families of exotic spheres. 

 
 
While previous work in stable homotopy theory has typically focused on exotic spheres in general, Kervaire and Milnor's work naturally divides exotic spheres into two subsets: those which bound parallelizable manifolds (``bP spheres") and those which do not (``very exotic spheres" \cite{Sch85}). From the perspectives of geometric topology and Riemannian geometry, very exotic spheres are particularly mysterious. For instance, every bP sphere is known to admit a smooth faithful $\mr{S}^1$-action \cite{HH67}, but not every very exotic sphere is know to admit such an action. As another example, only one very exotic sphere is known to admit a Riemannian metric of positive Ricci curvature, whereas all  bP spheres admit such metrics \cite{Wra97}. 

Every even-dimensional exotic sphere is necessarily very exotic, whereas most of the known odd-dimensional exotic spheres are bP spheres. Kervaire and Milnor's work shows that identifying very exotic spheres is fundamentally a problem of detecting elements  in the cokernel of the $\mr{J}$-homomorphism. The work mentioned above \cite{BHHM20, BMQ23} implies that very exotic spheres exist in at least $93$ congruence classes of dimensions modulo $192$. The 6 new congruence classes in \Cref{MT:Families} (all except for $74$) are not covered by prior work, leading to the following conclusion. 



\begin{corollary*}
\emph{Very exotic spheres exist in over half of all dimensions. More precisely, they exist in at least $99$ congruence classes of dimensions modulo $192$.} 
\end{corollary*}

This leads us to suggest a variant of the Wang--Xu conjecture. The low-dimensional computations of Isaksen--Wang--Xu \cite{IWX23} and Ravenel \cite{Rav86}, together with the recent results \cite{BHHM20,BMQ23}, verify the  following conjecture  up to dimension $102$. 

\begin{conjecture}
\emph{Very exotic spheres exist in all dimensions greater than $4$, except dimensions $5$, $6$, $11$, $12$, $27$, $43$, $56$, and $61$. }
\end{conjecture}

\subsection{Methodology} We consider  a type $2$  spectrum $\A_1$ which is constructed using three cofiber sequences 
\begin{equation} \label{C1}
\begin{tikzcd}
\SS \rar["2"] & \SS \rar & \M \rar["{\sf p}_1"] & \Sigma \SS, 
\end{tikzcd}
\end{equation}
\begin{equation} \label{C2}
\begin{tikzcd}
\Sigma\M \rar["\eta"] & \M \rar & \Y  \rar["{\sf p}_2"] & \Sigma^2 \M,
\end{tikzcd}
\end{equation}
\begin{equation} \label{C3}
\begin{tikzcd}
\Sigma^2 \Y \rar["v"] & \Y \rar & \A_1  \rar["{\sf p}_3"] & \Sigma^3 \Y, 
\end{tikzcd}
\end{equation}
where $v$ is a choice of a $v_1^1$-self-map of $\Y$. The recent work of Viet-Cuong Pham \cite{Pha23}, which shows that the  $\tmf$-Hurewicz homomorphism 
\begin{equation} \label{HA1}
\begin{tikzcd}
{\sf h}_{\tmf}: \pi_*\A_1 \rar[two heads] & \tmf_*\A_1
\end{tikzcd}
\end{equation}
is a surjection, is the starting point of our calculations. We then study long exact sequences associated to the cofiber sequences  \eqref{C1}, \eqref{C2} and \eqref{C3} using  our knowledge  of $\tmf_*$ \cite{Bau08, DFHH14, BR21}, as well as $\tmf_*\M$, $\tmf_*\Y$, and $\tmf_*\A_1$  \cite{BBPX22, Pha23}.

By combining this study  with our complete knowledge of the Hurewicz image in $\tmf_*$ \cite{BMQ23},  we identify seven new infinite families of elements in  $\pi_*\SS$ (listed  in \Cref{MT:Families}) which are in the image of the pinch map 
\begin{equation} \label{eq:p3p2p1}
\begin{tikzcd}
{\sf p}: \A_1 \rar["{\sf p}_3"] & \Sigma^3 \Y \rar["{\sf p}_2"] & \Sigma^5 \M \rar["{\sf p}_1"] & \Sigma^6 \SS 
\end{tikzcd}
\end{equation}
in stable homotopy. 


The $192$-periodic elements in the stable stems constructed in \cite{BMQ23} were all shown to have order at most $8$. The $\tmf$-homology calculations of \Cref{SS:lift}  lead to  new information about the $2$-torsion of some of the $192$-periodic infinite families identified  in \cite{BMQ23}. We deduce this from \Cref{Table:A1lifts} using the fact that the elements in the image of $p_1$ are simple $2$-torsion, where $p_1$ is the map defined in \eqref{tmfLES3}. 

 \begin{thmx}\label{MT:tordiv}
 \emph{ There exists a simple $2$-torsion element in the stable stems with $\tmf$-Hurewicz image} $\Delta^{8k}x$, $k \geq 0$, where
  \[ x \in \{ \kappa \nu, \ 4\overline{\kappa}, \ \overline{\kappa}^2\eta^2, \ \eta \Delta \overline{\kappa}^2, \ 4 \Delta^2 \overline{\kappa}, \ \overline{\kappa}^4, \  \eta^2 \Delta^2 \overline{\kappa}^2, \ 2 \Delta^4 \cdot 2 \overline{\kappa}, \ 4 \Delta^6 \overline{\kappa}\}.\]
  \end{thmx}

\begin{remark}
One can alternatively show that $\Delta^{8k} \kappa \nu$ and $\Delta^{8k}\bar{\kappa}^2 \eta^2$ admit simple $2$-torsion lifts in the stable stems using the facts that $\Delta^{8k} \kappa \nu = (\Delta^{8k} \nu) \cdot \kappa$, $\Delta^{8k}\bar{\kappa}^2\eta^2 = (\Delta^{8k}\bar{\kappa}^2 \eta) \cdot \eta$, $\Delta^{8k}\nu$ and $\Delta^{8k}\bar{\kappa}^2 \eta$ are in the $\tmf$-Hurewicz image, and $\kappa$ and $\eta$ are simple $2$-torsion in the stable stems. We do not know similar factorizations of the other elements of \cref{MT:tordiv}, so our computations provide the only proof we know that these elements admit simple $2$-torsion lifts. 
\end{remark}

\subsection*{Organization of the paper} In \Cref{SS:lift}, we perform the technical  $\tmf$-homology calculations which are necessary in 
\Cref{Sec:A1}  to prove  \Cref{MT:Families} and \Cref{MT:KT2}. 

While reading this paper, the reader may  find \cite[Part I, Ch. 12]{DFHH14} convenient for looking up the homotopy groups of $\tmf$, where the generators in the Hurewicz image are marked with colored dots.  We refer to Figures 8, 9, 21 and 22 in \cite{BBPX22} for explicit descriptions of $\tmf_*\M$ and $\tmf_*\Y$. 

\subsection*{Acknowledgments}
The authors would like to thank Mark Behrens, Guozhen Wang, and Zhouli Xu for helpful discussions, Bob Bruner and Dan Isaksen for their predictions in \Cref{rmk:predict}, and  Bert Guillou and Haynes Miller for pointing out consequential typos and for their comments on earlier versions of this paper. We also thank the anonymous referees for carefully reading our paper and for their feedback, suggestions and corrections. 

The research for this paper has been supported by the National Science Foundation through the grants DMS-2039316, DMS-2135884, DMS-2239362, DMS-2305016, DMS-2414922, and DMS-2441241.

\section{Some $\tmf$-homology calculations}  \label{SS:lift}

Using our knowledge of $\tmf_*$ \cite{Bau08, DFHH14}, $\tmf_*\M$, $\tmf_*\Y$, and $\tmf_*\A_1$ \cite{BBPX22, Pha23} we will compute the maps $i_k$ and $p_k$ in the long exact sequences
\begin{equation} \label{tmfLES1}
\begin{tikzcd}
\cdots \arrow{r} & \tmf_{k} \Y \rar["i_3"] &\tmf_{k} \A_1 \rar["p_3"] & \tmf_{k-3}\Y \rar["v_*"] & \cdots.
\end{tikzcd}
\end{equation}
\begin{equation} \label{tmfLES2}
\begin{tikzcd}
\cdots \arrow{r} & \tmf_{k-3} \M \rar["i_2"] &\tmf_{k-3} \Y \rar["p_2"] & \tmf_{k-5}\M \rar["\eta_*"] & \cdots.
\end{tikzcd}
\end{equation}
\begin{equation} \label{tmfLES3}
\begin{tikzcd}
\cdots \arrow{r} & \tmf_{k-5}  \rar["i_1"] &\tmf_{k-5} \M \rar["p_1"] & \tmf_{k-6} \rar["2"] & \cdots 
\end{tikzcd}
\end{equation}
associated to the cofiber sequences \eqref{C1}, \eqref{C2} and \eqref{C3}, respectively.  This is the technical core of the paper and requires   careful bookkeeping using Adams--Novikov spectral sequences. 

An element $y \in \tmf_{k-3} \Y$ is in the image of $p_3$ for some version of $\A_1$ if and only if \[ v_1 \cdot y = 0 \in \tmf_{k-1} \Y\]  for a choice of $v_1$. The action of all choices of $v_1$   has been identified on generators of  $\tmf_*\Y$ in  \cite[Figs. 22, 23]{BBPX22}. Thus,  the image of $p_3$ is easily determined; we list these elements in the leftmost column of  \cref{Table:A1lifts}.
 \begin{remark} We note that the ambiguity surrounding the action of $v_1$ on $\tmf_*\Y$, as presented in \cite[Remark 6.42]{BBPX22}, does not impact the computed list of $v_1$-torsion elements in $\tmf_*\Y$. 
 \end{remark}
\subsection{$v_1$-periodic families} \ 

We define 
\begin{equation} \label{tmfKOv1}
 v_1^{-1} \tmf  := {\colim} \  \{ \tmf \overset{c_4} \longrightarrow \Sigma^{-8} \tmf \overset{c_4} \longrightarrow \Sigma^{-16} \tmf \overset{c_4} \longrightarrow \dots  \} \simeq \mr{KO}[\mathfrak{j}^{-1}],     
 \end{equation}
where $\mathfrak{j}^{-1} = \Delta/ c_4^3$  and $\mr{KO}$ is the $2$-local periodic real $\mr{K}$-theory. The  equivalence above follows from  \cite[Corollary 3]{Lau04}. 
\begin{remark} The paper \cite{Lau04} establishes that the spectrum  $v_1^{-1} \tmf $ is $\mr{K}(1)$-locally equivalent to $\mr{KO}[\mathfrak{j}^{-1}]$, or equivalently, the equivalence of \eqref{tmfKOv1} holds after a $2$-adic completion. However,  the claimed $2$-local equivalence of \eqref{tmfKOv1} can be deduced from \cite[Corollary 3]{Lau04} by applying the arithmetic fracture square of \cite[Proposition 2.9]{BousfieldLoc}. The difference between $2$-local and $2$-complete settings is immaterial for our  subsequent calculations. We choose to proceed with the $2$-local settings in order to be consistent with our use of elliptic spectral sequences (see \Cref{notn:elliptic}) which converge to $2$-local $\tmf$-homology groups. 
\end{remark}

\begin{definition}  \label{defn:v1tor} For any spectrum $\mr{X}$,  define the $v_1$-torsion part of  $\tmf_* \mr{X}$ as the kernel 
\[ 
\tmf_* (\mr{X})^{\sf tor} := 
\ker \left(
\begin{tikzcd}
\ell: \tmf_* \mr{X} \rar & v_1^{-1} \tmf_* \mr{X}
\end{tikzcd}
\right)
 \]
 of the $v_1$-localization map. 
\end{definition}

\begin{definition} For any spectrum $\mr{X}$,  define the $v_1$-periodic part of  $\tmf_*\mr{X}$ as the cokernel 
\[ 
\tmf_* (\mr{X})^{\sf per} := \coker \left(
\begin{tikzcd}
   \tmf_* (\mr{X})^{\sf tor} \rar[hook] & \tmf_* \mr{X}
\end{tikzcd}
\right)
 \]
 of the natural  inclusion map.
\end{definition}

The homotopy groups $\tmf_*$, $\tmf_*\mr{M}$, and $\tmf_*\mr{Y}$ have nontrivial $v_1$-torsion and $v_1$-periodic parts. In $\tmf_*$, the $v_1$-periodic part consists of those elements which are $c_4$-torsion free; these appear, for example, near the horizontal axis in the picture on \cite[Pg. 190]{DFHH14}. In $\tmf_*\mr{M}$, the $v_1$-periodic part consists of the families of elements represented by open circles in \cite[Figs. 8, 9]{BBPX22}, and in $\tmf_*\mr{Y}$, the $v_1$-periodic part consists of all elements in filtration zero in \cite[Figs. 21, 22]{BBPX22}.

 On the other hand, the $v_1$-periodic part of $\tmf_*\mr{A}_1$ is trivial, since $\mr{A}_1$ is a type 2 spectrum. Hence, any element in $\tmf_*\mr{A}_1$ is $v_1$-torsion. 

Since $\tmf_*\mr{A}_1$ is the starting point of our computations,  the classes we are interested in are actually in the $v_1$-torsion part, so we will ignore the $v_1$-periodic part of the homotopy in \eqref{tmfLES1}. The discussion in the next few pages makes this precise. 

\begin{lemma} \label{lem:map-to-tor} For any nonzero element $a \in \tmf_* \A_1$  we have
\begin{enumerate}[(a)]
\item $p_3(a) \in \tmf_*(\Y)^{\sf tor}$, 
\item $p_2(p_3(a)) \in \tmf_*(\M)^{\sf tor}$, and 
\item $p_1(p_2(p_3(a)) \in \tmf_*(\SS)^{\sf tor}$.
\end{enumerate} 
\end{lemma}
\begin{proof}
	The statement of this lemma follows from the observation that $v_1$-torsion elements cannot map to $v_1$-periodic elements. Since  any $a\in \tmf_*\mr{A}_1$ is $v_1$-torsion, we have that $p_3(a)$ must be a $v_1$-torsion element. We repeat this argument to prove parts (b) and (c). 
%
%
\end{proof}
\subsection{From $\tmf_*\Y$ to $\tmf_*\M$} \label[section]{sec22}   \ 

We begin our computations by considering the $v_1$-periodic classes in the long exact sequence \eqref{tmfLES2}.  Utilizing the known equivalences  $v_1^{-1} \tmf \simeq \mr{KO}[\mathfrak{j}^{-1}]$ of \eqref{tmfKOv1}, and  the fact that $\mr{KO} \sma  \Y \simeq \mr{K}(1)$ \cite{DMv1v2}, we conclude:
\begin{equation} \label{v1tmfY} 
v_1^{-1}\tmf_* \Y \cong \mathbb{F}_2[v_1^{\pm 1},  \Delta]. 
\end{equation}
By comparing  this with the $v_1$-periodic part of $\tmf_*\Y$ found in the zero line of \cite[Figs. 21, 22]{BBPX22}, we conclude: 
\[\tmf_*\Y^{\sf per}\cong \mathbb{F}_2[v_1, \Delta^8] \{1, \Delta v_1,  \Delta^2 v_1^2, \Delta^3 v_1^3,  \Delta^4v_1, \Delta^5v_1^2, \Delta^6v_1^3, \Delta^7v_1^4\}.\]
Likewise, using the known calculation of $\mr{KO}_*\M$ we have 
\begin{equation} \label{v1tmfM}
 v_1^{-1}\tmf_* \M \cong \mr{KO}_*\M [\Delta]
 \end{equation}
 and comparing it with \cite[Figs. 8, 9]{BBPX22} one can easily deduce the  $v_1$-periodic part of $\tmf_*\M$, namely $\tmf_*\M^{\sf per}$. 
For example, the $v_1$-periodic part of $\tmf_*\M$ consists of classes in filtrations 0,1 and 2 in \cite[Figs. 8, 9]{BBPX22} which belong to a series of typical ``lightning flash'' patterns. However,  describing $\tmf_*\M^{\sf per}$ with a closed formula is challenging.  Specifically, while every class denoted by an open circle is $v_1$-periodic, certain other classes (such as $\Delta^2 \eta^2$) represented by solid bullets in the same figures are  also $v_1$-periodic.  Our precise calculations are detailed in the subsequent two remarks.

\begin{remark} As an $\mathbb{F}_2$-vector space,  $\mr{KO}_*\M \cong \mr{L}[v_1^{\pm4}]$, where $\mr{L}$  is the graded vector space constituting the  ``lightning flash" pattern: 
	
	\begin{figure}[h] 
		\begin{center}
			\begin{tikzpicture}[scale = .6, 
				open circle/.style={circle, draw=black, fill=white, inner sep=2pt}
				]
				\draw (0,0) -- (1,1) -- (2,2) -- (2,0) -- (3,1) -- (4,2);
				\node[open circle, label={ left: $\iota$}]  at (0,0) {};
				\node[open circle, label={ left: $\eta \cdot \iota$}] at (1,1 ) {};
				\node[open circle, , label={ left: $\eta^2 \cdot \iota$}] at (2,2) {};
				\node[open circle, label={ below: $v_1 \cdot \iota$}] at (2,0) {};
				\node[open circle, label={ right: $v_1 \eta \cdot \iota$}]  at (3,1) {};
				\node[open circle,  label={ right: $v_1 \eta^2 \cdot \iota$}] at  (4,2) {};
			\end{tikzpicture}
		\end{center}
		\vspace{-10pt}
		\caption{The lightning flash pattern $\mr{L}$}
	\end{figure}
	where $v_1 \cdot \iota$ is the Toda bracket $\langle \eta, 2, \iota \rangle$. 
	Thus, one may describe $v_1^{-1}\tmf_* \M$ as $\mr{L}[v_1^{\pm 4}][\Delta]$. 
\end{remark}
 \begin{remark} \label{rem:peranomaly}
By comparing  \eqref{v1tmfM} with \cite[Figs. 8, 9]{BBPX22}, we observe that the entire lightning flash pattern on generators $\Delta^i v_1^{4k}$ lifts along the injection 
\[ 
\begin{tikzcd} 
\tmf_*\M^{\sf per} \rar[hook] & v_1^{-1}\tmf_*\M
\end{tikzcd}
\]
provided $k \geq 1$. When $k=0$,  only a fraction of the lightning flash pattern on powers of $\Delta$ lift to  $\tmf_*\M^{\sf per}$. This pattern is periodic when multiplied by $\Delta^8$. 

Since $v_1$-torsion elements do not map to $v_1$-periodic elements, we have a commutative diagram 
\begin{equation} \label{eqn:pertorexact}
\begin{tikzcd}
\dots \rar &  \tmf_{k-3}\M^{\sf tor} \rar["i_2"] \dar[hook]  & \tmf_{k-3} \Y^{\sf tor} \rar["p_2"] \dar[hook] & \tmf_{k-5} \M^{\sf tor} \rar \dar[hook]  & \dots  \\
\dots \rar &  \tmf_{k-3}\M \rar["i_2"] \dar[two heads]  & \tmf_{k-3} \Y \rar[ "p_2"] \dar[two heads] & \tmf_{k-5} \M \rar \dar[two heads]  & \dots  \\
\dots \rar &  \tmf_{k-3}\M^{\sf per} \rar["i_2"'] & \tmf_{k-3} \Y^{\sf per} \rar["p_2"'] & \tmf_{k-5} \M^{\sf per} \rar & \dots  
\end{tikzcd}
\end{equation}
where the vertical sequences are short exact, and among the horizontal sequences, only  the one in the middle is long exact. 

\end{remark}

\begin{lemma}\label[lemma]{lem:new}
In the $v_1$-periodic sequence given by the bottom row of  \eqref{eqn:pertorexact}
\[	\begin{tikzcd}
\cdots \ar[r, "\eta_{\ast-1}"] & \tmf_{\ast} \M^{\sf per} \rar["i_2"] &\tmf_{\ast} \Y^{\sf per} \rar["p_2"] & \tmf_{\ast -2}\M^{\sf per} \rar["\eta_{\ast-2}"] & \cdots,
\end{tikzcd} \]
the maps $i_2$ and $p_2$ are determined by the following equations whenever the element in the domain is defined:
\begin{align*}
	i_2(\Delta^n v_1^{4k})&=\Delta^{n} v_1^{4k} \\
	i_2(\Delta^n v_1^{4k+1})&=\Delta^{n} v_1^{4k+1} \\
	&p_2(\Delta^n v_1^{4k+2}) = \Delta^n v_1^{4k}\eta^2 \\
	&p_2(\Delta^n v_1^{4k+3})= \Delta^n v_1^{4k+1} \eta^2
\end{align*}
Consequently,
\begin{itemize}
\item[-] $\mr{img}(p_2) = \ker(\eta_{\ast-2})$,
\item [-] $ \mr{img}(i_2) = \ker (p_2)$,
\item[-] and the  elements  of $\tmf_*\M^{\sf per}$ in 
\[ \mr{E}^\M := \{ \Delta \eta, \Delta^2 \eta^2, \Delta^2 v_1 \eta,  \Delta^3 v_1 \eta^2, \Delta^4 \eta,  \Delta^5 \eta^2, \Delta^5 v_1 \eta, \Delta^6 v_1 \eta^2 \} \]
and their $\Delta^8$-multiples are neither in the image of $\eta_*$ nor map to a nonzero  element under $i_2$. 
\end{itemize}
\end{lemma}
\begin{proof} 
Using \eqref{v1tmfY}, \eqref{v1tmfM}, and our knowledge of  the $\mr{KO}$-homology long exact sequence associated to \eqref{C2}, one observes the behavior within the long exact sequence obtained by $v_1$-localizing \eqref{tmfLES2}:
\[ 
	\begin{tikzcd}
		\cdots \arrow{r} &v_1^{-1} \tmf_{\ast} \M \rar["i_2"] &v_1^{-1}\tmf_{\ast} \Y \rar["p_2"] & v_1^{-1}\tmf_{\ast-2}\M \arrow{r} &\cdots
	\end{tikzcd} 
\]
The maps are explicitly given by:
	\begin{align*}
	i_2(v_1^{4k}\Delta^i)&=v_1^{4k} \Delta^i \\
	i_2(v_1^{4k+1}\Delta^i) &=v_1^{4k+1} \Delta^i\\
	&p_2(v_1^{4k+2} \Delta^i) = v_1^{4k}\eta^2 \Delta^i\\
	&p_2(v_1^{4k+3} \Delta^i) =v_1^{4k+1} \eta^2 \Delta^i.
	\end{align*}
Combining this with our observations in \Cref{rem:peranomaly}, we get the result. 
\end{proof}

We say that an element of $\tmf_*(\mr{X})$ is $v_1$-periodic if it is not an element of the $v_1$-torsion submodule (as defined in \Cref{defn:v1tor}). A $v_1$-periodic element is precisely one that maps to a nonzero element in the $v_1$-periodic quotient module, $\tmf_*(\mr{X})^{\sf per}$. It is important to keep in mind that the set of $v_1$-periodic elements does not constitute a submodule.
\begin{notn}
Let ${\bf E}^\M := \mathbb{F}_2 \{ x \cdot \Delta^{8i}: x \in \mr{E}^\M \text{ and } i \in \NN \} $.
\end{notn}
\begin{definition}[Exceptional $v_1$-periodic elements]
 \label{defn:exception1}	
We call a $v_1$-periodic element of $\tmf_*\M$ \emph{exceptional} in \eqref{tmfLES2} if its image under the quotient map $\tmf_*\M \twoheadrightarrow \tmf_*\M^{\sf per}$  is a nonzero element of ${\bf E}^{\M}$. Otherwise, we call it a \emph{non-exceptional} $v_1$-periodic element. 
\end{definition}
\begin{thm} \label{exceptional} \emph{ A $v_1$-periodic element $m \in \tmf_*\M$  is exceptional in \eqref{tmfLES2}  if and only if $i_2(m)$ is a  nonzero $v_1$-torsion element.}
\end{thm}
\begin{proof} Suppose $m$ is a $v_1$-periodic element in $\tmf_*\M$. Then, the result follows from analyzing the commutative diagram with surjective vertical maps: 
\[ 
\begin{tikzcd}
\tmf_{*-1}\M \rar[ "\eta_{*-1}"] \dar[two heads, "q"'] & \tmf_*\M \rar["i_2"] \dar[two heads, "q"'] & \tmf_*\Y \dar[two heads, "q'"] \\
\tmf_{*-1}\M^{\sf per} \rar[ "\eta_{*-1}"]  & \tmf_*\M^{\sf per} \rar["i_2"] & \tmf_*\Y^{\sf per}
\end{tikzcd}
\]
($\Rightarrow$) Suppose  $m$ is an exceptional $v_1$-periodic element. Then $q (m) \in {\bf E}^{\M}$ and  \[ q' (i_2(m)) = i_2(q(m)) = 0\]
by \Cref{lem:new}.    Therefore, $i_2(m) \in \mr{ker}(q') = \tmf_*\Y^{\sf tor} $. Furthermore, if  $i_2(m) = 0$ then it would imply that $ m = \eta_*(m')$ for some $v_1$-periodic element $m'$. Consequently, $q(m)$ will be in the image of $i_2$, which will contradict \Cref{lem:new}. 

($\Leftarrow$) Conversely, assume $m$ is a non-exceptional $v_1$-periodic element.  Then $q'(i_2(m)) = i_2(q(m)) \neq 0$ by \Cref{lem:new}. Thus $i_2(m)$ is a non-exceptional $v_1$-periodic element of $\tmf_*\Y$. 
\end{proof}

Now we are set to determine  the map $p_2$ on the classes within  $\mr{img}(p_3)$, which is a subspace of  $ \tmf_{*} \Y^{\sf tor}$ by  \Cref{lem:map-to-tor}. 
We  use the long exact sequence \eqref{tmfLES2} and the corresponding short exact sequence 
\begin{equation}\label{eqn:SESY}
\begin{tikzcd}
  \mr{C}_{k-3} \rar[hook, "i_2"] &\tmf_{k-3} \Y \rar[two heads, "p_2"] & \mr{K}_{k-5},
\end{tikzcd}
\end{equation}
 where $\mr{C}_{\ast}: = \tmf_{\ast}\M/ \mr{img}(\eta_*)$ is the cokernel of $\eta_*$ in \eqref{tmfLES2}, and 
$ \mr{K}_{\ast}= \ker (\eta_*) \subseteq \tmf_\ast \M$ is the kernel of $\eta_* $ in \eqref{tmfLES2}.

\begin{notn} We define $\mr{K}_{\ast}^{\sf tor}$ and  $\mr{C}_{\ast}^{\sf tor}$ as the kernel and the cokernel, respectively,  of the map $\eta_*$ restricted to  $\tmf_*\M^{\sf tor}$. By definition,  $\mr{K}_{\ast}^{\sf tor}$ is a subset of $ \mr{K}_{\ast}$ and $\mr{C}_{\ast}^{\sf tor}$ is a subset of $\mr{C}_{\ast}$.  
\end{notn}
Restricting \label{eqn:SESY} to $v_1$-torsion elements, we do not get an exact sequence. The failure is precisely due to the existence of exceptional $v_1$-periodic elements in \eqref{tmfLES2} which map to $v_1$-torsion elements. Therefore:
\begin{notn}
 We define $\mr{C}_{\ast}^{e}$ as the span of the $v_1$-torsion and exceptional $v_1$-periodic elements of $\mr{C}_{\ast}$ in \eqref{tmfLES2}. 
\end{notn} 
By \Cref{lem:new} and \Cref{exceptional}, we get a short exact sequence 
\begin{equation} \label{SESe}
\begin{tikzcd}
\mr{C}_{k-3}^{e} \rar[hook] &  \tmf_{*}\Y^{\sf tor} \rar[two heads] & \mr{K}_{k-5}^{\sf tor}
\end{tikzcd}
\end{equation}
which allows us to ignore non-exceptional  $v_1$-periodic elements  in \eqref{tmfLES2} in  the subsequent calculations. 
\begin{notn} \label{notn:elliptic}
In \cite{BBPX22}, the $\tmf$-homology of $\Y$ and $\M$ is calculated using the Adams--Novikov spectral sequence (ANSS):
\[ {^{(-)}\mr{E}}_2^{ s, t} := \mr{Ext}_{\Gamma}^{s,t} \left(\mr{A}, \pi_*(\tmf \wedge \mr{X}(4) \wedge (-) )\right) \implies \tmf_{t -s } (-), \label{eqn:ANSS}\] 
where the spectrum $\mr{X}(4)$ and the Hopf algebroid $(\mr{A}, \Gamma)$ are defined as in \cite[$\mathsection$2.1]{BBPX22}. We define the Adams--Novikov filtration of an element $x \in \tmf_* (\mr{X})$ to be $s$, denoted $\mr{AF}(x) = s$, if it is detected by an element
\[ \widehat{x} \in {^{\mr{X}}\mr{E}}_{2}^{s, * +s} \]
on the $\mr{E}_2$-page of the ANSS \eqref{eqn:ANSS}.

 We use $ {\sf s}_{ i,j}$, ${\sf m}_{ i,j}$, and $ {\sf y}_{i,j}$ to denote elements in $\tmf_*$, $\tmf_*\M$, and $\tmf_*\Y$, respectively, that are detected in ${\sf E}_{\infty}^{j, j+i}$ of the Adams--Novikov spectral sequence \eqref{eqn:ANSS},
 i.e. the elements ${\sf s}_{ i,j}$, ${\sf m}_{ i,j}$, and ${\sf y}_{ i,j}$ have Adams--Novikov filtration $j$ and stem $i$.
\end{notn}

\begin{remark}
In the bidegrees of interest, the only nonzero element present is either a $v_1$-torsion element or an exceptional $v_1$-periodic element. Consequently, $ {\sf s}_{ i,j}$, ${\sf m}_{ i,j}$, and $ {\sf y}_{ i,j}$ each represent a unique element up to a higher Adams--Novikov filtration.
\end{remark}

We rely on the calculation of $\tmf_*{\M}$ presented in \cite[Figs. 8, 9]{BBPX22} and employ several standard techniques in our analysis, which are  listed below: 
\begin{linearity}  \label{T3} The maps $i_2$ and $p_2$ in  \eqref{tmfLES2} and \eqref{eqn:SESY} are $\tmf_*$-linear, i.e., 
	\begin{enumerate}
		\item  $p_2(t \cdot y ) = t \cdot p_2(y)$,
		\item  $i_2 (t \cdot m) = t \cdot i_2(m)$ 
	\end{enumerate}
	for all $t \in \tmf_* $, 
	$m \in \tmf_* \M$   and $y \in \tmf_* \Y$. 
	In particular, we will often set $t$  as $\kappabar \in \pi_{20}\tmf$, and exploit the $\kappabar$-linearity of the  maps $i_2$ and $p_2$.
\end{linearity}
 
\begin{technique}[Vanishing $\mr{K}^{\sf tor}$] \label{T1}  If $\Kmr_{k-5}^{\sf tor} =0$ in  \eqref{eqn:SESY}, then \[ p_2(y) = 0\]
 for any $y \in \tmf_{k-3}(\Y)^{\sf tor}$. 
\end{technique}
\begin{proof}
Since a $v_1$-torsion element  never maps to a $v_1$-periodic element, we have a commutative diagram: 
\[ 
\begin{tikzcd}
\tmf_*\Y^{\sf tor} \rar["p_2"] \dar[hook] & \mr{K}^{\sf tor}_{*-2} \dar[hook] \\
\tmf_*\Y \rar[two heads, "p_2"] & \mr{K}_{*-2}.
\end{tikzcd}
\]
The result is a direct consequence of this. 
\end{proof}

\begin{application} 
We employ \Cref{T1} to  conclude that the  following elements  map to zero under the map $p_2: \tmf_{k-3}\Y \to \tmf_{k-5}\M$:
 \begin{multicols}{4}
 \begin{itemize}
 \item ${\sf y}_{3,1}$
 \item  ${\sf y}_{6, 2}$
 \item  $ {\sf y}_{14,2}$
 \item  ${\sf y}_{18,2}$
  \item ${\sf y}_{21,3}$
 \item ${\sf y}_{29,5}$
 \item ${\sf y}_{34,6}$
 \item ${\sf y}_{39,7}$
 \item  ${\sf y}_{40,6}$
  \item ${\sf y}_{51,1}$
  \item ${\sf y}_{54,2}$
 \item  ${\sf y}_{60,10}$
\item  ${\sf y}_{60,12}$
\item ${\sf y}_{65,7}$
\item ${\sf y}_{65,13}$
\item ${\sf y}_{66,2}$
\item ${\sf y}_{69,3}$
\item ${\sf y}_{75,13}$
\item ${\sf y}_{76,10}$
\item ${\sf y}_{80,16}$
\item ${\sf y}_{81,11}$
\item ${\sf y}_{86,12}$
\item ${\sf y}_{90,14}$
\item ${\sf y}_{91,13}$
\item ${\sf y}_{97,9}$

\item ${\sf y}_{112,12}$
\item ${\sf y}_{117,3}$
\item ${\sf y}_{123,11}$
 \item ${\sf y}_{150,2}$
 \item ${\sf y}_{161,7}$
 \item ${\sf y}_{165,3}$
 \end{itemize}
 \end{multicols}
Additionally, $\kappabar$-linearity implies that the following classes map to zero under the map $p_2: \tmf_{k-3}\Y \to \tmf_{k-5}\M$:
\begin{multicols}{4}
	\begin{itemize}
		\item ${\sf y}_{85,17}$
		 \item ${\sf y}_{96,14}$ 
		 	 \item ${\sf y}_{101,15}$
			 \item ${\sf y}_{105,21}$ 
\item ${\sf y}_{106,16}$
		 	  \item ${\sf y}_{111,17}$
\item ${\sf y}_{116,18}$
		 	  	 \item ${\sf y}_{117,13}$
		 	  	 \item ${\sf y}_{132,16}$
		 	 \item ${\sf y}_{137,17}$
		 \item ${\sf y}_{143,15}$	 
	\end{itemize}
	\end{multicols}
  \end{application}
 \begin{technique}[Vanishing $\mr{C}$] \label{T2}
 If $y \in\tmf_{k-3}( \Y)^{\sf tor}$ is a nonzero element and $\mr{C}_{k-3}^e = 0$,  then  \[ p_2(y) \neq 0\] 
 in \eqref{eqn:SESY}.
Further, if $\mr{rank}_{\FF_2}(\mr{K}^{\sf tor}_{k-5}) = 1$, then the image of  $y$ is  the unique nonzero element of $\mr{K}^{\sf tor}_{k-5}$. 
\end{technique}
\begin{application} We employ \Cref{T2} to  determine the following: 
 \begin{multicols}{2}
 \begin{itemize}
 \item $p_2({\sf y}_{8,2}) = {\sf m}_{6,2}$
 \item  $p_2({\sf y}_{11,3}) = {\sf m}_{9,3}$
 \item  $p_2({\sf y}_{23,3}) = {\sf m}_{21,3}$
 \item  $p_2({\sf y}_{26,4}) = {\sf m}_{24,6}$
 \item  $p_2({\sf y}_{44,8}) = {\sf m}_{42,10}$
 \item  $p_2({\sf y}_{59,3}) = {\sf m}_{57,3}$
 \item  $p_2({\sf y}_{62,2}) = {\sf m}_{60,12}$
  \item  $p_2({\sf y}_{74,4}) = {\sf m}_{72,6}$
  \item  $p_2({\sf y}_{77,5}) = {\sf m}_{75,13}$
  \item $p_2({\sf y}_{83,3}) = {\sf m}_{81,3}$
  \item  $p_2({\sf y}_{87,7}) = {\sf m}_{85,13}$
  	\item  $p_2({\sf y}_{88,6}) = {\sf m}_{86,12}$
    \item  $p_2({\sf y}_{92,8}) = {\sf m}_{90,10}$
    \item  $p_2({\sf y}_{93,3}) = {\sf m}_{91,9}$
     \item  $p_2({\sf y}_{98,4}) = {\sf m}_{96,6}$
    \item  $p_2({\sf y}_{119,3}) = {\sf m}_{117,3}$
  \item  $p_2({\sf y}_{127,15}) = {\sf m}_{125,21}$      
             
   \item  $p_2({\sf y}_{155,3}) = {\sf m}_{153,3}$
     \item  $p_2({\sf y}_{167,3}) = {\sf m}_{165,3}$
      \item  $p_2({\sf y}_{170,4}) = {\sf m}_{168,6}$
 \end{itemize}
 \end{multicols}
Together with $\kappabar$-linearity, these imply:
 \begin{multicols}{2}
 	\begin{itemize}
 		\item  $p_2({\sf y}_{82,6}) = {\sf m}_{80,16}$
  		\item $p_2({\sf y}_{103,7}) = {\sf m}_{101,7}$ 
 		\item $p_2({\sf y}_{113,7}) = {\sf m}_{111,13}$ 
  		\item $p_2({\sf y}_{118,8}) = {\sf m}_{116,10}$
	\item $p_2({\sf y}_{133,11}) = {\sf m}_{131, 17}$		
 		 \item $p_2({\sf y}_{138,12}) = {\sf m}_{136,14}$
 		 		\item $p_2({\sf y}_{153,15}) = {\sf m}_{151, 21}$, hence  \\ $p_2({\sf y}_{153,11}) = 0$ 
 		\item  $p_2({\sf y}_{158,16}) = {\sf m}_{156,18}$
 	\end{itemize}
 	\end{multicols}
\end{application}

Our next technique follows from the fact that maps of Adams--Novikov spectral sequences induced by maps of spectra  cannot decrease filtration.

\begin{technique}[Adams--Novikov filtration argument] \label{T5} 
	Consider the short exact sequence \eqref{SESe}.
	\begin{enumerate}[leftmargin=*]
		\item 
	Suppose that in \eqref{SESe} we have  
	\[\mr{C}_{k-3}^e \cong \FF_2\{\m_{k-3, a}\}, \ \ \  \tmf_{k-3}\Y^{\sf tor}  \cong \FF_2\{\y_{k-3, b}, \y_{k-3, c}\}, \ \ \  \mr{K}_{k-5}^{\sf tor} \cong \FF_2\{\m_{k-5, d}\},\] 
	where $a>c$ and $b>d$. Then
	\[i_2(\m_{k-3, a})=\y_{k-3, b} \qquad \text{and} \qquad p_2(\y_{k-3, c})=\m_{k-5, d}.\] 	
\item 
  Suppose that in \eqref{SESe} we have 
  \[\mr{C}_{k-3}^e=0, \ \ \ \tmf_{k-3}\Y^{\sf tor}=\FF_2\{{\sf y}_{k-3, a}, {\sf y}_{k-3, b}\},\ \ \ \mr{K}_{k-5}^{\sf tor}=\FF_2\{{\sf m}_{k-5, a}, {\sf m}_{k-5, c}\},\] 
  where $c\geq b \geq a$. Then 
	\[p_2(\y_{k-3, b})=\m_{k-5, c}  \qquad \text{ and } \qquad p_2(\y_{k-3, a})=\y_{k-3, a}.\]
		\end{enumerate}
\end{technique}
	\begin{proof}  The argument relies on the fundamental property that the maps in the short exact sequence $\eqref{SESe}$ are induced by maps of filtered objects and thus cannot decrease the filtration degree. 
	
{\bf Part (1)}: The constraint $a>c$ forces the highest filtration identification: $i_2({\sf m}_{k-3, a})= {\sf y}_{k-3, b}$. This implies $p_2({\sf y}_{k-3, b})=0$. The remaining required surjectivity for the exact sequence, coupled with the constraint $b>d$, establishes the second relation: ${\sf m}_{k-5, d}=p_2({\sf y}_{k-3, c})$.

{\bf  Part (2)}: With $\mr{C}_{k-3}^e=0$ and $c \geq b$, the map $p_2$ must send ${\sf y}_{k-3, b}$ to an element of Adams-Novikov filtration greater than or equal to $b$. Thus, unambiguously $$p_2({\sf y}_{k-3, b})={\sf m}_{k-5, c}.$$

Then we have  either $p_2({\sf y}_{k-3, a})= {\sf m}_{k-5, a}$ or $p_2({\sf y}_{k-3, a})= {\sf m}_{k-5, a}+ {\sf m}_{k-5, c}$. In the second case, we simply adjust the  basis for $\mr{K}_{k-5}^{\sf tor}$ so that $p_2({\sf y}_{k-3, a})= {\sf m}_{k-5, a}$.  
This basis adjustment is valid because ${\sf m}_{k-5, a}$ is only uniquely determined up to elements of strictly higher filtration.
 \end{proof} 
\begin{application}  \ 
	We use part (1) of \Cref{T5}  to conclude that:
	\begin{itemize}
		\begin{multicols}{2}
		\item  $p_2({\sf y}_{6, 2}) = 0$, 
				\item  $p_2({\sf y}_{20, 2}) = m_{18,2}$, 
				\item $p_2({\sf y}_{35, 3}) = {\sf m}_{33,3}$,
				\item  $p_2({\sf y}_{45,3}) = {\sf m}_{43,9}$ and\\ $p_2({\sf y}_{45,9}) = 0$,
\item $p_2({\sf y}_{55,7}) = {\sf m}_{53,7}$, 
\item  $p_2({\sf y}_{57,11}) = 0$,	
\item  $p_2({\sf y}_{71,9})=0$ and \\$p_2({\sf y}_{71,3}) = {\sf m}_{69,3}$,
\item $p_2({\sf y}_{102,10}) = {\sf m}_{100,20}$ and \\$p_2({\sf y}_{102,2}) = 0$
	\item $p_2({\sf y}_{122,14}) = 0$ and \\$p_2({\sf y}_{122,4}) = {\sf m}_{120,6}$. 
\end{multicols}
	\end{itemize}
We use part (2) of  \Cref{T5}  to conclude that:	
\begin{itemize}
	\begin{multicols}{2}
\item  $p_2({\sf y}_{56,2}) =  {\sf m}_{54,2}$ and \\
$p_2({\sf y}_{56,6}) =  {\sf m}_{54,6}$
\item $p_2({\sf y}_{68, 2}) = {\sf m}_{66,2}$ 
	\item $p_2({\sf y}_{107,11}) = {\sf m}_{105,17}$  and\\ $p_2({\sf y}_{107,3}) = {\sf m}_{105,3}$
	\item $p_2({\sf y}_{108, 10}) = {\sf m}_{106,16}$
\end{multicols}
\end{itemize}
Using $\kappabar$-linearity, these imply:
\begin{itemize}
		\begin{multicols}{2}
				\item $p_2({\sf y}_{128,14}) =\m_{126,20}$
	\item $p_2({\sf y}_{142,18}) =0$
		\item $p_2({\sf y}_{148,18}) =0$.
		\item $p_2({\sf y}_{168,22}) = 0$.
		\end{multicols}
\end{itemize}	
\end{application}

\begin{technique}[Extended $\bar{\kappa}$-linearity argument]\label{T6}
In the long exact sequence \eqref{tmfLES2}, we have:
\begin{itemize}
	\item $p_2(\y_{50,6})=0$ and  $p_2(\y_{50,4})=\m_{48,6}$
	\item $p_2(\y_{70,10})=0$ and  $p_2(\y_{70,8})=\m_{68,10}$
\end{itemize}
\end{technique}

\begin{proof}	We consider the short exact sequence \eqref{SESe} 
		\begin{equation}
		\begin{tikzcd}
			\mr{C}_{50}^{e} \rar[hook, "i_2"] &  \tmf_{50}\Y^{\sf tor} \rar[two heads, "p_2"] &  \mr{K}_{48}^{\sf tor}, 
		\end{tikzcd}
	\end{equation}
	where  $\mr{C}_{50}^e=\FF_2\{\m_{50,6}\}$, $\tmf_{50}\Y^{\sf tor}=\FF_2\{\y_{50,4}, \y_{50,6}\}$, and  $ \mr{K}_{48}^{\sf tor}=\FF_2\{\m_{48,6}\}$.
	
To determine the map $p_2$, we utilize the relations $\y_{50,6}=\kappabar  \cdot \y_{30,2}$ and $p_2(\y_{30,2})=\m_{28,6}$. By $\kappabar$-linearity,  we 
have $p_2(\y_{50,6})=\kappabar \cdot  \m_{28,6}.$ Since $\kappabar$ has Adams--Novikov filtration $4$, and $\m_{28,6}$ has filtration $6$, the class $\kappabar \cdot \m_{28,6}$ must have filtration 10 or greater. Since $\mr{K}_{48}^{\sf tor}$ is one-dimensional and $p_2$ is surjective, the remaining basis element must map to the generator. We define ${\sf y}_{50,4}$ to be the element in $\tmf_{50}\Y^{\sf tor}$ such that $p_2({\sf y}_{50,4}) = {\sf m}_{48,6}$. This choice is unique up to adding a multiple of ${\sf y}_{50,6}$, which corresponds to modifying ${\sf y}_{50,4}$ by an element of higher Adams–Novikov filtration.
	 
The second case follows from the first case by $\kappabar$-linearity.
\end{proof}

\subsection{ From $\tmf_*\M$ to $\tmf_*$} \

In this subsection we compute the maps in the long exact sequence \eqref{tmfLES3}
	\begin{equation*} 
	\begin{tikzcd}
		\cdots \arrow{r} & \tmf_{k-5}  \rar["i_1"] &\tmf_{k-5} \M \rar["p_1"] & \tmf_{k-6} \rar["2"] & \cdots 
	\end{tikzcd}
\end{equation*}
As a first step, we prove an analogue of \Cref{lem:new} for \eqref{tmfLES3} and then investigate the behavior of the $v_1$-periodic elements in \eqref{tmfLES1}.

\begin{remark} \label{rem:peranomaly2}
The $v_1$-periodic structure of $\tmf_*$ is explicitly characterized by comparison with its $v_1$-localization, where we have the isomorphism:
\begin{equation} \label{v1tmfS}
\pi_*(v_1^{-1}\tmf) \cong v_1^{-1} \tmf_* \cong \mr{KO}_*[\mathfrak{j}^{-1}].
\end{equation}
Any nonzero element of the form $t\Delta^n c_6$ and $t \Delta^n c_4^k \eta^{\delta}$, where $t \in \mathbb{Z}_{(2)}$, $\delta \in \{ 0, 1, 2 \}$ and $k \geq 0$, is $v_1$-periodic. These elements map, respectively, to $t \Delta^n v_1^6$ and $t \Delta^{n} v_1^{4k} \eta^{\delta}$ under the localization map
\[ 
\begin{tikzcd}
\tmf_*^{\sf per} \rar[hook] & \mr{KO}_*[\mathfrak{j}^{-1}]. 
\end{tikzcd}
\]
\end{remark}

\begin{lemma}\label{lemma:new2}
In  the $v_1$-periodic sequence associated to \eqref{tmfLES3}
\[  \begin{tikzcd}
\cdots \rar["2"] & \tmf_{k-5}^{\sf per}  \rar["i_1"] &\tmf_{k-5} \M^{\sf per} \rar["p_1"] & \tmf_{k-6}^{\sf per} \rar["2"] & \cdots 
\end{tikzcd}
\]
the maps $i_1$ and $p_1$ are determined by the following equations whenever the element in the domain is defined
	\begin{align*}
		i_1(\Delta^n c_4^{k}\eta^{\delta})&=\Delta^{n} v_1^{4k} \eta^{\delta} \\
		i_1(2\Delta^n c_4^{k-1}c_6)&=\Delta^{n} v_1^{4k+1}\eta^2 \\
		&p_1(\Delta^n v_1^{4k+1}\eta^{\epsilon-1}) = \Delta^n c_4^{k} \eta^{\epsilon} 
	\end{align*}
for $\delta \in \{0,1,2\}$, $\epsilon \in \{1,2\}$, and $k \geq 1$. Consequently:
\begin{itemize}
\item[-] $\mr{img}(p_1) = \ker(2)$, 
\item[-] the  elements  of $\tmf_*^{\sf per}$ in 
\[ \mr{F}^\mathbb{S} := \{ 8\Delta, 4\Delta^2, 8\Delta^3,  2\Delta^4, 8 \Delta^5, 4\Delta^6, 8\Delta^7 \} \]
and their $\Delta^8$-multiples are neither in the image of (multiplication by) $2$ map nor  map to a nonzero  element under $i_1$, and
 \item[-] the nonzero elements of $\tmf_* \M^{\sf per}$ in the span
  \[ \mr{F}^{\M}:= \{ v_1\eta^2, \Delta v_1\eta^2, \Delta^2 v_1 \eta^2, \Delta^3 v_1 \eta^2, \Delta^4 v_1 \eta^2, \Delta^5 v_1 \eta^2, \Delta^6 v_1 \eta^2  \} \]
and their $\Delta^8$-multiples are neither in the image of $i_1$ nor  map to a nonzero element under $p_1$. 
\end{itemize}
\end{lemma}
\begin{proof}
Using \eqref{v1tmfM}, \eqref{v1tmfS}, and our knowledge of the $\mr{KO}$-homology long exact sequence associated to \eqref{C1},  one observes the behavior within the long exact sequence obtained by $v_1$-localizing \eqref{tmfLES1}:
\[ 
	\begin{tikzcd}
		\cdots \arrow{r} &v_1^{-1} \tmf_{\ast}  \rar["i_2"] &v_1^{-1}\tmf_{\ast} \M \rar["p_2"] & v_1^{-1}\tmf_{\ast-1} \arrow{r} &\cdots
	\end{tikzcd} 
\]
The maps are explicitly given by:
	\begin{align*}
		i_1(\Delta^n c_4^{k}\eta^{\delta})&=\Delta^{n} v_1^{4k} \eta^{\delta} \\
		i_1(2\Delta^n c_4^{k-1}c_6)&=\Delta^{n} v_1^{4k+1}\eta^2 \\
		&p_1(\Delta^n v_1^{4k+1}\eta^{\epsilon-1}) = \Delta^n c_4^{k} \eta^{\epsilon} 
	\end{align*}
for $\delta \in \{0,1,2\}$ and $\epsilon \in \{1,2\}$. 

Combining this with our observations in \Cref{rem:peranomaly} and \Cref{rem:peranomaly2}, we get the result. 
\end{proof}

Thus we set ${\bf F}^\mathbb{S} := \mathbb{F}_2 \{ x \cdot \Delta^{8i}: x \in \mr{F}^\mathbb{S} \text{ and } i \in \NN \} $ and  ${\bf F}^\M := \mathbb{F}_2 \{ x \cdot \Delta^{8i}: x \in \mr{F}^\M \text{ and } i \in \NN \} $, and make the following definition analogous to \Cref{defn:exception1}. 
\begin{definition}[Exceptional $v_1$-periodic elements] \label{defn:exception2}
We call a $v_1$-periodic element of $\tmf_*$ (or $\tmf_*\M$) \emph{exceptional} in \eqref{tmfLES1} if its image under the quotient map $\tmf_* \twoheadrightarrow \tmf_*^{\sf per}$ (or $\tmf_*\M \twoheadrightarrow \tmf_*\M^{\sf per}$)  is a nonzero element of ${\bf F}^{\mathbb{S}}$ (or ${\bf F}^{\M}$). Otherwise, we call it a \emph{non-exceptional} $v_1$-periodic element. 
\end{definition}

Arguments identical to those in \Cref{exceptional} show that only the exceptional $v_1$-periodic elements map to $v_1$-torsion elements in \eqref{tmfLES2}. Consequently, we get an exact sequence (similar to \eqref{SESe})
\begin{equation}\label{SESee}
\begin{tikzcd}
 \mr{C}_{k-5}^{e} \rar[hook] &  \tmf_{k-5}(\M)^{e} \rar[two heads] &  \mr{K}_{k-6}^{\sf tor}, 
\end{tikzcd}
\end{equation}
where 
\begin{itemize}
\item[-] $\mr{C}_{\ast}^{e}$ is the span of the images of $v_1$-torsion and exceptional $v_1$-periodic elements of $\coker(\tmf_* \overset{2}{\to} \tmf_*)$ in \eqref{tmfLES2},
\item[-] $\tmf_{\ast}(\M)^{e}$ is the span of the images of $v_1$-torsion and exceptional $v_1$-periodic elements of $\tmf_*\M$ in \eqref{tmfLES2}, and 
\item[-]  $\mr{K}_{\ast}^{\sf tor}$ is the kernel of the multiplication by $2$ map on $\tmf_*$. 
\end{itemize}
Moreover, \Cref{T1}, \Cref{T2}, and \Cref{T5}  from  \Cref{sec22} have  analogs for addressing \eqref{SESee}. We  use them to get the following results: 
 
\begin{lemma}
 We use the analog of \Cref{T1} to determine:
\begin{multicols}{2}
 \begin{itemize}
 \item $i_1({\sf s}_{6,2}) = {\sf m}_{6,2}$
 \item $i_1({\sf s}_{24,0}) = {\sf m}_{24,6}$
 \item $i_1({\sf s}_{48,0}) = {\sf m}_{48,6}$
 \item  $i_1({\sf s}_{57,3}) = {\sf m}_{57,3}$
  \item $i_1({\sf s}_{72,0}) = {\sf m}_{72,6}$
  \item $i_1({\sf s}_{75,3}) = {\sf m}_{75,13}$
  \item $i_1({\sf s}_{80,16}) = {\sf m}_{80,16}$
 \item  $i_1({\sf s}_{85,13}) = {\sf m}_{85,13}$
 \item $i_1({\sf s}_{90,10}) = {\sf m}_{90,10}$
 \item $i_1({\sf s}_{96,0}) = {\sf m}_{96,6}$
  \item $i_1({\sf s}_{120,0}) = {\sf m}_{120,6}$
  \item $i_1({\sf s}_{153,3} ) = {\sf m}_{153,3}$
  \item $i_1({\sf s}_{168,0} ) = {\sf m}_{168,6}$.
 \end{itemize}
 \end{multicols}
 These, together with $\kappabar$-linearity, imply:
 \begin{itemize}
	\begin{multicols}{2}
  		  \item $i_1({\sf s}_{68,4}) = {\sf m}_{68,10}$ 
  		   \item $i_1({\sf s}_{100,20}) = {\sf m}_{100,20}$
  		      \item $i_1({\sf s}_{116,4}) = {\sf m}_{116,10}$
  		       \item $i_1({\sf s}_{136,8}) = {\sf m}_{136,14}$
  		      \item $i_1({\sf s}_{156,12}) = {\sf m}_{156,18}$
	\end{multicols}
 \end{itemize}
 \end{lemma}
 
 \begin{lemma}
 We use the analog of \Cref{T2} to determine:
\begin{multicols}{2}
 \begin{itemize}
 \item $p_1({\sf m}_{18,2}) = {\sf s}_{17,2}$
 \item $p_1({\sf m}_{43,9}) = {\sf s}_{42,10}$
 \item $p_1({\sf m}_{69,3}) = {\sf s}_{68,4}$
 \item  $p_1({\sf m}_{81,3}) = {\sf s}_{80,16}$
 \item $p_1({\sf m}_{86,12}) = {\sf s}_{85,13}$
 \item $p_1({\sf m}_{91,9}) = {\sf s}_{90,10}$
 \item $p_1({\sf m}_{165,3}) = {\sf s}_{164,4}$. 
 \end{itemize}
 \end{multicols}
 These, together with $\kappabar$-linearity, imply:
 \begin{itemize}
 	\begin{multicols}{2}
 	 \item $p_1({\sf m}_{101,7}) = {\sf s}_{100,20}$	
 	  \item $p_1({\sf m}_{106,16}) = {\sf s}_{105,17}$
 	  \item $p_1({\sf m}_{111,13}) = {\sf s}_{110,14}$
	   	  \item  $p_1({\sf m}_{126,20}) = {\sf m}_{125, 21}$ 
 	  \item  $p_1({\sf m}_{131,17}) = {\sf s}_{130,18}$
 	  \item  $p_1({\sf m}_{151,21}) = {\sf s}_{150,22}$
 	\end{multicols}	
 \end{itemize}		
 \end{lemma}

 \begin{lemma} We use the analog of \Cref{T5} to deduce:
 \begin{itemize}
 \item $i_1({\sf s}_{9,3}) \neq {\sf m}_{9,1}$, which forces $i_1({\sf s}_{9,3}) = {\sf m}_{9,3}$, 
  \item $i_1({\sf s}_{21,5}) \neq {\sf m}_{21,3}$, which forces $p_1({\sf m}_{21,3}) = {\sf s}_{20,4}$, 
 \item $p_1({\sf m}_{33,3}) \neq {\sf s}_{32,2}$, which forces $ i_1({\sf s}_{33,3}) = {\sf m}_{33,3}$,  
 \item $i_1({\sf s}_{42,10}) \neq {\sf m}_{42,8}$, which forces $i_1({\sf s}_{42,10}) = {\sf m}_{42,10}$, 
 \item $i_1({\sf s}_{54,2})\neq {\sf m}_{54,6}$, which forces $i_1({\sf s}_{54,2})={\sf m}_{54,2}$,
  \item $i_1({\sf s}_{60,12}) \neq {\sf m}_{60,7}$, which forces $i_1({\sf m}_{60,12}) = {\sf s}_{60,12}$, 
 \item $i_1({\sf s}_{66,10}) = {\sf m}_{66,10}$, which forces $p_1({\sf m}_{66,2}) = {\sf s}_{65,3}$ 
   		     \item $i_1({\sf s}_{105, 17}) =\m_{105, 17}$, which forces  $i_1({\sf s}_{105, 3}) ={\sf m}_{105, 3}$
 \item $i_1({\sf s}_{117,5}) \neq {\sf m}_{117, 3}$, which forces  $p_1({\sf m}_{117,3}) = {\sf s}_{116,4}$.
 \end{itemize}
 These, together with $\kappabar$-linearity imply:
 \begin{itemize}
 	\begin{multicols}{2}
		\item $i_1({\sf s}_{53,7}) = {\sf m}_{53,7}$
		\item  $i_1({\sf s}_{125,21}) = {\sf m}_{125,21}$
	\end{multicols}
\end{itemize} 
 \end{lemma}
 
\subsection{Summary Table} \ 
We summarize our calculations in \Cref{Table:A1lifts} as follows. The leftmost column lists the image of $p_3$ in $\tmf_*\Y$. We determine their image in column $2$ and indicate the technique used, among \Cref{T1} through \Cref{T6}, in column $3$.  

We calculate the image  under $p_1$ of nonzero elements in column $2$ and record them in column $4$. If the image is zero, we identify a $v_1$-torsion element which is its lift along $i_1$ and record it in column $5$. We indicate the technique in column $6$. 

Note that the elements listed in columns $4$ and $5$ are elements of $\tmf_*$. We record their familiar names from \cite{DFHH14} in column $7$.

\begin{longtable}{| l || l l | l l l | c |   }
\caption{Detecting elements in $\tmf_*$}  \label{Table:A1lifts} \\
\toprule
$\mr{img}(p_3)$   & $\mr{img}(p_2)$ & (T)  & $\mr{img}(p_1)$   &  $i_1^{-1}(-)$ & (T)  & name in $\tmf_*$    \\
\midrule \endhead
\bottomrule \endfoot
${\sf y}_{3,1}$ & $0$ & (1)  &  &  &&  \\
${\sf y}_{6,2}$ & $0$ &  (1) &  & & & \\
${\sf y}_{8,2}$ & ${\sf m}_{6,2}$ &  (2) & $0$ & ${\sf s}_{6,2}$  & (1)  & $\nu^2$  \\
${\sf y}_{11,3}$ & ${\sf m}_{9,3}$ &  (2) &$0$  & ${\sf s}_{9,3}$  &  (3) & $\nu^3$  \\
${\sf y}_{14,2}$ & $0$ & (1) &&&& \\
${\sf y}_{18,2}$ & $0$ &  (1) &  &&& \\
${\sf y}_{20,2}$ & ${\sf m}_{18,2}$ &    (3) & ${\sf s}_{17,2}$ & &  (2)&  $\kappa \nu$   \\
${\sf y}_{21,3}$ & $0$ &  (1) &&&& \\
${\sf y}_{23,3}$ & ${\sf m}_{21,3}$ &   (2) & ${\sf s}_{20,4}$  &  & (3)& $4\overline{\kappa}$ \\
${\sf y}_{26,4}$ & ${\sf m}_{24,6}$ &  (2)& $0$ &  ${\sf s}_{24,0}$   &   (1)    & $8 \Delta$\\
${\sf y}_{29,5}$ & $0$ &  (1) &&&& \\
${\sf y}_{34,6}$ & $0$ &  (1)   &&&& \\
${\sf y}_{35,3}$ & ${\sf m}_{33,3}$  &   (3) &0 & ${\sf s}_{33,3}$   & (3)& $q\eta$ \\
${\sf y}_{39,7}$ & $0$ &   (1)   &&&&  \\
${\sf y}_{40,6}$ & $0$ & (1)  &&&&\\
${\sf y}_{44,8}$ & ${\sf m}_{42,10}$ &   (2)  &$0$  & ${\sf s}_{42,10}$ &  (3)& $\overline{\kappa}^2 \eta^2$  \\
${\sf y}_{45,3}$ & ${\sf m}_{43,9}$ &  (3)  & ${\sf s}_{42,10}$ &&   (2)& $\overline{\kappa}^2 \eta^2$\\
${\sf y}_{45,9}$ & $0$ & (3) &&&&  \\
${\sf y}_{50,4}$ &  ${\sf m}_{48,6}$ & (4)&$0$   &${\sf s}_{48,0}$ &(1) & $4 \Delta^2$ \\
${\sf y}_{50,6}$ & $0$ & (4) & &  &  &    \\
${\sf y}_{51,1}$ & $0$ & (1) && & &  \\
${\sf y}_{54,2}$ & $0$ & (1)    && & & \\
${\sf y}_{55,7}$ & ${\sf m}_{53,7}$  & (3) & $0$ & ${\sf s}_{53,7}$ &  (3) & $\eta^2 \Delta^2 \nu$  \\
${\sf y}_{56,2}$ & ${\sf m}_{54,2}$ & (3)  & 0 & ${\sf s}_{54,2}$  & (3) &   $\nu \Delta^2 \nu$ \\
${\sf y}_{57,11}$ & $0$ & (3) & &  && \\ 
${\sf y}_{59,3}$ & ${\sf m}_{57,3}$ & (2) & 0  &${\sf s}_{57,3}$ & (1) & $\nu \Delta^2 \nu^2$ \\ 
${\sf y}_{60,10} $ & $0$ & (1) & &  & &  \\
${\sf y}_{60,12} $ & $0$ & (1) & &  & &  \\
${\sf y}_{62,2}$ & ${\sf m}_{60,12}$ &(2) &  & ${\sf s}_{60,12}$ & (3)&  $\nu \Delta^2 \nu^3$  \\
${\sf y}_{65,7}$ & $0$ &(1) &&&&\\
${\sf y}_{65,13}$ & $0$ &(1) &&&&\\
${\sf y}_{66,2}$ & $0$ &(1) &&&&\\
${\sf y}_{68,2}$ & ${\sf m}_{66,2}$  &(3) & ${\sf s}_{65,3}$ & &  (3) &  $\eta\Delta \bar{\kappa}^2$  \\
${\sf y}_{69,3}$ & $0$ &(1) &  & & & \\
${\sf y}_{70,8}$ & $ {\sf m}_{68,10}$ &(4) & 0 & ${\sf s}_{68,4}$ & (1) & $4 \Delta^2 \bar{\kappa}$  \\
${\sf y}_{70,10}$ & $ 0$ &(4)  & & & &\\
${\sf y}_{71,3}$ & $ {\sf m}_{69,3}$ &(3) &  ${\sf s}_{68,4}$ & 0 &  (2) & $4 \Delta^2 \bar{\kappa}$ \\
${\sf y}_{71,9}$ & $0 $ & (3) && &&  \\
${\sf y}_{74,4}$ & ${\sf m}_{72,6}$ & (2) & $0$ & ${\sf s}_{72,0}$ & (1) &$8 \Delta^3$\\ 
${\sf y}_{75,13}$ & $0$ &(1)&&&&\\
${\sf y}_{76,10}$ & $0$ &(1)&&&&\\ 
${\sf y}_{77,5}$ & ${\sf m}_{75,13}$ &(2) & $0$ & ${\sf s}_{75,3}$ &(1) & $(\eta\Delta)^3 $ \\
${\sf y}_{80,16}$ & $0$ &(1) &&&& \\
${\sf y}_{81,11}$ & $0$ &(1) &&&&\\
${\sf y}_{82,6}$ & ${\sf m}_{80,16}$ &(2) & $0$ & ${\sf s}_{80,16}$ & (1) & $\kappabar^4$ \\
${\sf y}_{83,3}$ & ${\sf m}_{81,3}$ & (2) & ${\sf s}_{80,16} $  & & (2) &  $\kappabar^4$  \\
${\sf y}_{85,17}$ & $0$ &(1) &&&&\\
${\sf y}_{86,12}$ & $0$ &(1) &&&& \\
${\sf y}_{87,7}$ & ${\sf m}_{85,13}$ &(2)& $0$ & ${\sf s}_{85,13}$ &(1) & $\eta \Delta \bar{\kappa}^3$  \\
${\sf y}_{88,6}$ & ${\sf m}_{86,12}$ &(2)& ${\sf s}_{85,13}$ &  &(2) & $\eta \Delta \bar{\kappa}^3$ \\
${\sf y}_{90,14}$ & $0$ &(1)&&&&\\
${\sf y}_{91,13}$ & $0$ &(1) &&&&\\
${\sf y}_{92,8}$ & ${\sf m}_{90,10}$ &(2)& $0$ & ${\sf s}_{90,10}$ &(1) & $\eta^2 \Delta^2 \bar{\kappa}^2$ \\
${\sf y}_{93,3}$ & ${\sf m}_{91,9}$ &(2) &${\sf s}_{90,10}$ & &  (2) &  $\eta^2 \Delta^2 \bar{\kappa}^2$\\
${\sf y}_{96,14}$ & $0$ &(1) &&&&\\
${\sf y}_{97,9}$ & $0$ &(1) &&&&\\
${\sf y}_{98,4}$ & ${\sf m}_{96,6}$ &(2) & $0$  & ${\sf s}_{96,0}$ &(1) & $2 \Delta^4$  \\
${\sf y}_{101,15}$ & $0$ &(1) && &  & \\
${\sf y}_{102,2}$ & $0$ & (3) &&&&\\
${\sf y}_{102,10}$ & ${\sf m}_{100,20}$ &(3) & $0$ &${\sf s}_{100,20}$ &(1) &$ \overline{\kappa}^5$  \\ 
${\sf y}_{103,7}$ & ${\sf m}_{101,7}$ &(2) & ${\sf s}_{100,20}$  & & (2) & $\bar{\kappa}^5$    \\
${\sf y}_{105,21}$ & $0$ &(1) &&&&\\
${\sf y}_{106,16}$ & $0$ &(1)&&&&\\
${\sf y}_{107,3}$ & ${\sf m}_{105,3}$ &(3) &$0$  &${\sf s}_{105, 3}$ & (3)& $\nu^3 \Delta^4$ \\
${\sf y}_{107,11}$ & ${\sf m}_{105,17}$ &(3)& $0$  & ${\sf s}_{105, 17} $&(3)& $\eta \Delta \overline{\kappa}^4$\\
${\sf y}_{108, 10}$  & ${\sf m}_{106, 16}$ & (3) &  ${\sf s}_{105,17}$ & & (2) & $\eta \Delta \overline{\kappa}^4$  \\
${\sf y}_{111,17}$ & $0$ &(1)&&&&\\
${\sf y}_{112,12}$ & $0$ &(1)&&&&\\
${\sf y}_{113,7}$ & ${\sf m}_{111,13}$ &(2) & ${\sf s}_{110,14}$ & & (2)&$\eta^2 \Delta^2 \bar{\kappa}^3$  \\
${\sf y}_{116,18}$ & $0$ &(1) &&&&\\
${\sf y}_{117,3}$ & $0$ &(1) &&&&\\
${\sf y}_{117,13}$ & $0$ &(1) &&&&\\
${\sf y}_{118,8}$ & ${\sf m}_{116,10}$ &(2) & $0$ &  ${\sf s}_{116,4}$ &  (1) & $2\Delta^4 \bar{\kappa}$  \\
${\sf y}_{119,3}$ & ${\sf m}_{117,3}$ &(2) & ${\sf s}_{116,4}$ & & (3) &  $2 \Delta^4 \cdot 2 \bar{\kappa}$    \\
${\sf y}_{122,4}$ & ${\sf m}_{120,6}$ &(3) & $0$ & ${\sf s}_{120,0}$ & (1) & $8 \Delta^{5}$  \\
${\sf y}_{122,14}$ & $0$ &(3) & &  &&   \\
${\sf y}_{123,11}$ & $0$ & (1) &&&&\\
${\sf y}_{127,15}$ & ${\sf m}_{125,21}$ &(2)& $0$ & ${\sf s}_{125,21}$ & (3) &  $\eta \Delta \bar{\kappa}^5$  \\
${\sf y}_{128,14}$ & ${\sf m}_{126,20}$ &(3) & ${\sf s}_{125,21}$ & & (2) & $\eta \Delta \bar{\kappa}^5$   \\
${\sf y}_{132,16}$ & $0$ & (1) &&&&\\
${\sf y}_{133,11}$ & ${\sf m}_{131,17}$ &\pcom{(2)} & ${\sf s}_{130,18}$ & &(2) & $\eta^2 \Delta^2 \bar{\kappa}^4$   \\
${\sf y}_{137,17}$ & $0$ &(1) &&&&\\
${\sf y}_{138,12}$ & ${\sf m}_{136,14}$ &(2) & $0$ &  ${\sf s}_{136,8}$  & (1)& $\eta^2 \Delta^5 \kappa$   \\
$y_{142,18}$ & $0$ &(3) &&&&\\
${\sf y}_{143,15}$ & $0$ &(1) &&&&\\
${\sf y}_{148,18}$ & $0$ & (3) &&&& \\
${\sf y}_{150,2}$ & $0$ &(1) &&&&\\
${\sf y}_{153,11}$ & $0$ & (2) &&&&\\
${\sf y}_{153,15}$ & ${\sf m}_{151, 21}$ &(2) &${\sf s}_{150, 22}$ && (2)& $\eta^2 \Delta^2 \bar{\kappa}^5$ \\
${\sf y}_{155,3}$ & ${\sf m}_{153,3}$& (2)& $0$  & ${\sf s}_{153,3}$ & (1) & $\nu \Delta^6 \nu^2$   \\
${\sf y}_{158,16}$ & ${\sf m}_{156,18}$ & (2) & $0$ & ${\sf s}_{156,12}$ &(1) & $\nu \Delta^6 \eta \epsilon$   \\
${\sf y}_{161,7}$ & $0$ &(1)&&&&\\
${\sf y}_{165,3}$ & $0$ &(1)&&&&\\
${\sf y}_{167,3}$ & ${\sf m}_{165,3}$  &(2) & ${\sf s}_{164,4}$ &  &(2)& $4\Delta^6 \bar{\kappa}$  \\
${\sf y}_{168,22}$ & $0$ &(3)&&&&\\
${\sf y}_{170,4}$ & ${\sf m}_{168,6}$ &(2) & $0$ & ${\sf s}_{168,0}$ & (1) & $8\Delta^7$  \\
\end{longtable}

 \section{New infinite families  }\label{Sec:A1}
The following commutative diagram shows the map of long exact sequences associated to the cofiber sequence \eqref{C3} and induced by the $\tmf$-Hurewicz map: 
\begin{equation} \label{LES3}
\begin{tikzcd}
\cdots \arrow{r} & \pi_{k} \Y \arrow{r}{i_3} \arrow{d}{{\sf h}_\tmf} & \pi_{k} \A_1 \arrow{r}{p_3} \arrow[d, ->>, "{\sf h}_\tmf"] & \pi_{k-3} \Y \rar["v_*"] \arrow{d}{{\sf h}_\tmf} & \cdots \\
\cdots \arrow{r} & \tmf_{k} \Y \arrow{r}{i_3} &\tmf_{k} \A_1 \rar["p_3"'] & \tmf_{k-3}\Y \rar["v_*"] & \cdots.
\end{tikzcd}
 \end{equation}
  
\begin{lemma} \label{lem:lift1}
 Any nonzero element  of the form  $p_3(a) \in \tmf_*\Y$  is in the Hurewicz image of $\tmf$. 
\end{lemma}
\begin{proof}
This is a direct consequence of the fact that the $\tmf$-Hurewicz map  for $\A_1$ is a surjection \cite{Pha23} in \eqref{LES3}. 
 \end{proof}
 
 Next, we study the commutative diagram of long exact sequences 
 \begin{equation} \label{LES2}
\begin{tikzcd}
\cdots \arrow{r} & \pi_{k-3} \M \arrow{r}{i_2} \arrow{d}{{\sf h}_\tmf} & \pi_{k-3} \Y \arrow{r}{p_2} \arrow[d, "{\sf h}_\tmf"] & \pi_{k-5} \M \rar["\eta_*"] \arrow{d}{{\sf h}_\tmf} & \cdots \\
\cdots \arrow{r} & \tmf_{k-3} \M \rar["i_2"'] &\tmf_{k-3} \Y \rar["p_2"'] & \tmf_{k-5}\M \rar["\eta_*"'] & \cdots.
\end{tikzcd}
 \end{equation}
associated to the cofiber sequence \eqref{C2}.  
\begin{lemma}  \label{lem:lift2} 
Any nonzero element  of the form  $p_2(p_3(a)) \in \tmf_*\M$  is in the Hurewicz image of $\tmf$. 
\end{lemma}

 \begin{proof} If $p_2(p_3(a)) \neq 0$ then, in particular, $p_3(a)\neq 0$. Thus, by \Cref{lem:lift1}, there exists \[ \tilde{y} \neq 0 \in \pi_*\Y\]  such that ${\sf h}_{\tmf}(\tilde{y}) = p_3(a)$. The result then follows from commutativity of \eqref{LES2}. 
 \end{proof}

 \begin{proposition}\label{rmk:injection}
The action of $\Delta^8$  is faithful on  $\tmf_*\A_1$, $\tmf_*\Y$, $\tmf_*\M$,   $\tmf_*$,  the cokernel of the $\tmf$-Hurewicz map, and also on the $\tmf$-Hurewicz image restricted to $\tmf_*^{\sf tor}$ except for nonzero integer multiples of $\nu$. 
 \end{proposition} 
\begin{proof}
The faithfulness of the action of $\Delta^8$ on each module follows from the literature: For $\tmf_*$, see \cite{Bau08}. For $\tmf_*\A_1$, see \cite{Pha23}. For $\tmf_*\Y$ and $\tmf_*\M$, see \cite{BBPX22}. For image and cokernel of the $\tmf$-Hurewicz map, see \cite[Theorem 1.2]{BMQ23}. 
    \end{proof}

 \subsection{Infinite families in $2$-local stable stems } \label{sec:proofinffamily} \ 
 \begin{convention}
For a nonzero class $x\in \tmf_* \mr{X}$ we will denote by $\widetilde{x}$ any class in $\pi_* \mr{X}$ with the property that ${\sf h}_{\tmf}(\widetilde{x})=x$.  
\end{convention}

 Our final step  is studying the commutative diagram of long exact sequences 
 \begin{equation} \label{LES1}
\begin{tikzcd}
\cdots \arrow{r} & \pi_{k-5} \SS \arrow{r}{i_1} \dar["{{\sf h}_\tmf}"'] & \pi_{k-5} \M \arrow{r}{p_1} \arrow[d, "{\sf h}_\tmf"'] & \pi_{k-6} \SS \rar["\cdot 2"] \arrow{d}{{\sf h}_\tmf} & \cdots \\
\cdots \arrow{r} & \tmf_{k-5}  \rar["i_1"'] &\tmf_{k-5} \M \rar["p_1"'] & \tmf_{k-6} \rar["\cdot 2"'] & \cdots
\end{tikzcd}
 \end{equation}
 associated to the cofiber sequence \eqref{C1}. 

Consider an element of the form $p_2(p_3(a))\in\tmf_{k-5}\mr{M}$ for some $a\in \tmf_k\mr{A}_1$. 

Suppose first that $ p_1(p_2(p_3(a)))  \neq 0$. Then it follows from \eqref{LES3}, \Cref{lem:lift2}, and \Cref{rmk:injection} that there is  a $192$-periodic infinite family 
\[ 
\{ \widetilde{\sf s}_{k-6 + 192i }\in \pi_{k-6 + 192i}(\SS): i \in \NN \} 
\]
such that 
\begin{enumerate}
\item  $ {\sf h}_\tmf(\widetilde{\sf s}_{k-6})  = p_1(p_2(p_3(a))) $, 
\item $  {\sf h}_\tmf(\widetilde{\sf s}_{k-6 + 192i}) \neq 0 $ for all $i \in \NN$. 
\end{enumerate}


Now we consider the case $p_1(p_2(p_3(a)))  = 0$.   
 \begin{thm}\label{thm:three-options} 
Let $a \in \tmf_{k} \A_1$ be an element such that the class $  m := p_2(p_3(a))$ is nonzero in $\tmf_{k-5}\M$ and satisfies $p_1(m) = 0$ in $\tmf_*$. Then $m$ admits a lift along $i_1$ in \eqref{tmfLES3}. 
\begin{enumerate}[(I)] 
\item  If $i_1^{-1}(m) \cap \im({\sf h}_{\tmf}) \cap \tmf_{k-5}^{\sf tor} \neq \emptyset, $  and $k-5>3$, then there exists a $192$-periodic infinite family of nonzero elements in the stable stems,$$\{ \widetilde{\sf s}_{k-5 + 192i }\in \pi_{k-5 + 192i}(\SS): i \in \NN \}$$such that $i_1({\sf h}_\tmf(\widetilde{\sf s}_{k-5})) = m$ and the $\tmf$-Hurewicz image ${\sf h}_\tmf(\widetilde{\sf s}_{k-5 + 192i})$ is nonzero for all $i \in \NN$.
\item If $i_1^{-1}(m) \cap \im({\sf h}_{\tmf}) = \emptyset,$ then there exists a $192$-periodic infinite family of elements in the stable stems$$\{ \bar{\sf s}_{k-6 + 192i }\in \pi_{k-6 + 192i}(\SS): i \in \NN \}$$such that $\bar{\sf s}_{k-6 + 192i} \neq 0$ and ${\sf h}_\tmf(\bar{\sf s}_{k-6 + 192i}) = 0$ for all $i \in \NN$.
\end{enumerate}
 \end{thm}
\begin{proof} 
First, let $m = p_2(p_3(a)) \in \tmf_{k-5}\M$. Since $m \neq 0$, the faithfulness of the $\Delta^8$ action on $\tmf_*\M$ (by \Cref{rmk:injection}) ensures that the classes $\Delta^{8i} \cdot m$ are all nonzero:
\[ p_2(p_3(\Delta^{8i} \cdot a)) = \Delta^{8i} \cdot m \neq 0. \]
Furthermore, by \Cref{lem:lift2}, this family admits a lift to a family of nonzero elements in the homotopy groups $\pi_* \M$:
\begin{equation} \label{mtilde}
\{ \widetilde{m}_{k -5 + 192i} \in \pi_{k-5 + 192i}(\M): i \in \NN \}
\end{equation}
such that ${\sf h}_\tmf(\widetilde{m}_{k -5 + 192i}) =\Delta^{8i} \cdot m $ for all $i \in \NN$.

\textbf{Case (I): $i_1^{-1}(m)$ lifts to the Hurewicz image.}
Suppose $m$ admits a lift $s_{k-5} \in \tmf_{k-5}$ along $i_1$ such that $s_{k-5} \in \im({\sf h}_\tmf) \cap \tmf_{k -5}^{\sf tor}$. By \Cref{rmk:injection}, the family $\Delta^{8i} \cdot s_{k-5}$ is also contained in the Hurewicz image. We can then define the desired infinite family of elements
\[ \{ \widetilde{\sf s}_{k-5 + 192i} \in \pi_{k-5 + 192i}(\SS): i \in \NN \} \]
such that $ {\sf h}_{\tmf}( \widetilde{\sf s}_{k-5 + 192i}) = \Delta^{8i} \cdot s_{k-5}$. This family satisfies the stated nonzero Hurewicz image property.

\textbf{Case (II): $i_1^{-1}(m)$ does not lift to the Hurewicz image.}
If no lift of $m$ along $i_1$ is contained in $\im({\sf h}_\tmf)$, the same must hold for $\Delta^{8i} \cdot m$ for all $i \in \NN$ (by inspection of the Hurewicz image in \cite{BMQ23}). This implies that the element $\widetilde{m}_{k -5 + 192i} \in \pi_{*} \M$ must map to a nonzero class in $\pi_{*-1} \SS$ under $p_1$.  

Therefore, the family
\[ \{ \overline{\sf s}_{k-6 + 192i} = p_1(\widetilde{m}_{k -5 + 192i}): i \in \NN \} \]
is the desired infinite family with trivial Hurewicz image.
 \end{proof}
 \bigskip
 \begin{proof}[\bf Proof of \Cref{MT:Families}]
For each degree $k \in \{ 29, 53, 77, 80, 101, 119, 173 \}$, \Cref{Table:A1lifts} provides an element $a_k \in \tmf_{k}\A_1$ which satisfies the primary conditions of \Cref{thm:three-options}:
\begin{enumerate}[(i)]
\item $m_{k-5} = p_2(p_3(a_k)) \neq 0$, and
\item $p_1(m_{k-5}) = 0$.
\end{enumerate}
The table further establishes the existence of a lift $s_{k-5} \in \tmf_{k-5}^{\sf tor}$ such that $i_1(s_{k-5}) = m_{k-5}$.

The degrees of these lifts, $k-5 \in \{ 24, 48, 72, 75, 96, 114, 168 \}$, are the degrees where the $\tmf$-Hurewicz image is known to be trivial. Therefore, the lift $s_{k-5}$ is not contained in the Hurewicz image, which is the condition required for Case (II) of \Cref{thm:three-options}. The existence of the seven infinite families follows immediately.
\end{proof}
\begin{remark} \label[remark]{rmk:HurP}
To summarize, the existence of the seven infinite families described in \Cref{MT:Families} is a direct consequence of the following two facts:
\begin{enumerate}
\item The elements
\[
8\Delta, 4\Delta^2, 8 \Delta^3, (\eta \Delta)^3, 2\Delta^4, 8 \Delta^5, 8\Delta^7
\]
are not contained in the Hurewicz image in $\tmf_*$.
\item These elements are precisely the lifts along $i_1$ of nonzero classes in $\im(p_2 \circ p_3: \tmf_*\A_1 \to \tmf_{*-5}\M)$. 
\end{enumerate}
\end{remark}

 \subsection{ Infinite families in $\mr{T}(2)$-local stable stems } \label{sec:T2} \  

The infinite families identified in \Cref{sec:proofinffamily} descend to the $\mr{T}(2)$-local stable stems by construction, stemming from the definition of $\mr{T}(2)$-localization. We provide a general outline of this argument in the current subsection. Given that the telescope conjecture is false \cite{BHLS}, this naturally raises the question of whether these elements also descend to the $\mr{K}(2)$-local stable stems. While a nontrivial image in the $\mr{K}(2)$-local stable stems (a result detailed in the next subsection) is sufficient to guarantee a nontrivial image in the $\mr{T}(2)$-local stable stems, we present the current sketch to emphasize that the nontrivial $\mr{T}(2)$-local image of our infinite family is independent of the explicit $\mr{K}(2)$-local calculations that follow.

Let $\overline{s} \in \pi_{k-6}(\mathbb{S})$ be an element of the infinite family constructed in \Cref{MT:Families}. From our proof in \Cref{sec:proofinffamily}, we know $\overline{s}$ must arise from a composite:
\[\begin{tikzcd}
\overline{s}: \Sigma^{k-5} \mathbb{S} \rar["\widetilde{a}"] & \Sigma^{-5}\mr{A}_1 \ar[rr, "p_1 \circ p_2 \circ p_3"] && \mathbb{S}.
\end{tikzcd}\] 
Here, the $\tmf$-Hurewicz image of the first map, $a := {\sf h}_\tmf (\widetilde{a})$, and the subsequent classes $p_3 (a)$ and $p_2(p_3 (a))$ are all nonzero. Furthermore, any Hurewicz lift $\widetilde{m}_i \in \pi_*\M$ of the class $\Delta^{8i} \cdot p_2(p_3 (a)) $ maps to a nontrivial class in $\pi_*(\mathbb{S})$ under the map $p_1$ for all $i \in \NN$. By defining $\overline{s}_i := p_{1}(\widetilde{m}_i)$ we thus obtain the infinite family associated to $\overline{s}$.

In \cite{BEM17}, it is established that $\mr{A}_1$ admits a $v_2^{32}$-self-map
\[ \begin{tikzcd} v: \Sigma^{192} \A_1 \rar & \A_1 \end{tikzcd} \] 
which is detected by $\Delta^{8} \in \tmf_* $. This map satisfies the relation ${\sf h}_{\tmf}( v^{i} \circ \widetilde{a}) = \Delta^{8i} \cdot a$ for all $i \in \NN$, where $v^{i}$ denotes the $i$-fold composition of $v$. Therefore, one can choose the lift $\widetilde{m}_i$ above to be the class $p_2 \circ p_3 \circ v^{i} \circ \widetilde{a}$. This observation implies that every element $\overline{s}$ in \Cref{MT:Families} has a nontrivial image under the map:
\[ \begin{tikzcd}
\alpha_{1 \ast}: \pi_*(\mathbb{S}) \rar & \pi_*(\Phi_{\mr{A}_1}(\mathbb{S})),
\end{tikzcd} \]
where $\Phi_{\mr{A}_1}(\mathbb{S})$ is the spectrum obtained by substituting $\mr{X} = \mr{A}_1$ and $\mr{E} = \mathbb{S}$ in the following general definition.

\begin{notn} For a spectrum $\mr{E}$ and a finite spectrum $\mr{X}$ equipped with a $v_n$--self-map $v: \Sigma^{|v|}\mr{X} \to \mr{X}$, we define the Bousfield-Kuhn functor $\Phi_{\mr{X}}(\mr{E})$ as the telescope:
\[\Phi_{\mr{X}}(\mr{E}) := \underset{\to}{\colim} \ \{ \mr{E}^{\mr{X}} \overset{{v}^*}\longrightarrow \Sigma^{- |{v}|} \mr{E}^{ \mr{X} }\overset{{v}^*}\longrightarrow \Sigma^{- 2|{v}|}\mr{E}^{ \mr{X} } \longrightarrow \dots \}.\] 
We denote the natural map from $\mr{E}$ to $\Phi_{\mr{X}}(\mr{E})$ by $\alpha$ (sometimes with subscripts, such as $\alpha_{1}$ above).
\end{notn}

 \begin{notn} \label{notn:M}
Let $\mr{M}({\sf i})$ denote the cofiber of multiplication by $2^{\sf i}$ on the sphere spectrum $\SS$. Let $\mr{M}({\sf i},{\sf j})$ denote the cofiber of the $v_1^{\sf j}$-self-map on $\mr{M}({\sf i})$.
\end{notn}
 Let ${\sf p}$ denote the composite $p_1 \circ p_2 \circ p_3$. Then:
  \begin{proposition} \label{prop:pinchM14} The map ${\sf p}: \Sigma^{-6} \mr{A}_1 \longrightarrow \SS$ factors through $\Sigma^{-10} \mr{M}(1,4)$.  
  \end{proposition}
  \begin{proof} Let ${\sf ko}$ denote the connective real $\mr{K}$-theory. Since  ${\sf ko}_6 \mr{M} \cong 0$ it follows that the composite 
  \[ 
  \begin{tikzcd}
  \Sigma^6 \mathbb{S} \rar[hook] & \Sigma^6 \Y \rar["v_1^3"] & \Y  \rar[two heads, "{\sf p}_2"] &  \Sigma^2 \mr{M} 
  \end{tikzcd}
  \]
  is nonzero in ${\sf ko}$-homology, and hence, in stable homotopy. Further, we have a commutative diagram 
\[ 
\begin{tikzcd}
\Sigma^6 \Y \rar["v_1"] \dar["{\sf p}_2"'] & \Sigma^4 \Y \dar["{\sf p}_2 \circ v_1^3 "] \\
\Sigma^8 {\mr{M}} \rar["v_1^4"'] & {\mr{M}} 
\end{tikzcd}
\]
which implies that there is a map $\Sigma^4 \mr{A}_1 \longrightarrow \mr{M}(1,4)$ which factors the pinch map of  $\Sigma^4 \mr{A}_1$ to its top cell. 
  \end{proof}
  A consequence of \Cref{prop:pinchM14} is that  we have a directed system 
  \begin{equation} \label{T2seq}
  \begin{tikzcd}
  \Sigma^{-6}\mr{A}_1 \rar \ar[dd, "{\sf p}"'] &  \Sigma^{-10} \mr{M}(1, 4) \rar \ar[ddl] &  \Sigma^{-18} \mr{M}(2, 8) \rar \ar[ddll] &  \dots  \\
   &&   \dots  \\
  \SS
  \end{tikzcd}
  \end{equation}
  of type $2$ spectra which is cofinal  among all type $2$ spectra admitting a `pinch' map to $\SS$. 
  \begin{notn} Let  $\Phi_{k }(-) $ denote   $\Phi_{\mr{V}_k }( - ) $, where  $\mr{V}_k$ is the $k$-th entry of the sequence  \eqref{T2seq}. Let $\alpha_k$ denote the natural map from $\mathbb{S}$ to $\Phi_{ k}(\mathbb{S})$. 
  \end{notn}
  
  \begin{thm}\label{MT:T2}
\emph{All elements listed in \Cref{MT:Families}  have nonzero images in the $\mr{T}(2)$-local stable stems. }
  \end{thm}
  \begin{proof}
By the standard theory of Bousfield-Kuhn functors (see \cite{Kuhn}) the $\mr{T}(2)$-local sphere spectrum is given by the inverse limit:
  \[ 
  \SS_{\mr{T}(2)} \simeq \underset{\leftarrow}\lim  \ \Phi_k(\SS). 
  \]
 We have already established that any element  $\overline{\sf s} \in \pi_*(\SS)$   listed in  \Cref{MT:Families}  has a nonzero image  under  the first map $\alpha_1$. The maps $\alpha_{k\ast}$ fit into the following diagram:   \[ 
  \begin{tikzcd} 
  \pi_*(\SS) \ar[dd,"\alpha_{1\ast}"'] \ar[rdd, "\alpha_{2\ast}"']  \ar[rrdd, ""']  \\
 & &  \dots  \\
   \pi_*(\Phi_1(\SS))   & \lar   \pi_*(\Phi_2(\SS))     & \lar  \pi_*(\Phi_3(\SS)) & \lar   \text{ } \cdots 
 \end{tikzcd}
 \]
 Since the image of $\overline{\sf s}$ is nonzero under $\alpha_1$, its image in the inverse limit $\underset{\leftarrow}\lim \ \pi_* \Phi_k(\SS)$ is also nonzero. The final result then follows from the fact that the natural map
  \[ 
 \begin{tikzcd} 
 \pi_* ( \SS_{\mr{T}(2)}) \cong \pi_*( \underset{ \leftarrow  }\lim \ {\Phi_k(\SS)}) \ar[r, two heads] & \underset{ \leftarrow  }\lim \   \pi_*  \Phi_k(\SS) 
 \end{tikzcd}
 \]
  is a surjection (with Milnor $\lim^1$ term as the kernel). 
  \end{proof}

\subsection{Infinite families in $\mr{K}(2)$-local stable stems} \label{sec:K2} \

\begin{notn}
For notational simplicity, we deviate from the standard convention and let $\mr{TMF}$ denote the $\mr{K}(2)$-localization of $\tmf$. 
\end{notn}
 Since $\Delta^8$ acts faithfully on $\tmf^{\sma}_{2}$,  the natural $\mr{K}(2)$-localization map
\[ 
\begin{tikzcd}
\ell: \tmf \rar & \mr{TMF}
\end{tikzcd}
\]
induces an injection on stable homotopy groups. As a result, the computation of the $\tmf$-Hurewicz image, as established in \cite[Theorem 1.2]{BMQ23}, effectively completes the calculation of the Hurewicz image of $\mr{TMF}$.

From the commutative diagram 
\[ 
\begin{tikzcd}
\pi_*\mathbb{S} \rar["{\sf h}_{\tmf}"] \dar & \tmf_* \dar \\
\pi_*\mathbb{S}_{\mr{K}(2)} \rar["{\sf h}_{\mr{TMF}}"] & \mr{TMF}_*
\end{tikzcd}
\]
we deduce that the elements of stable stems with a nontrivial $\tmf$-Hurewicz image remain nonzero after $\mr{K}(2)$-localization. However, this  argument fails for the elements listed in \Cref{MT:Families} because their $\tmf$-Hurewicz image is trivial. Therefore, to establish that these elements are nontrivial after $\mr{K}(2)$-localization, we must employ an argument very similar to that in the proof of \Cref{MT:Families} in \Cref{sec:proofinffamily}. A careful analysis is warranted, as the image of the $\mr{K}(2)$-local Hurewicz map
\begin{equation} \label{K2Hur}
\begin{tikzcd}
{\sf h}_{\mr{TMF}}: \pi_*\mathbb{S}_{\mr{K}(2)} \rar[] & \mr{TMF}_*
\end{tikzcd}
\end{equation}
is potentially larger than the Hurewicz image of $\mr{TMF}$.

\begin{notn} For any spectrum $\mr{X}$, let $\widehat{\mr{X}}$ denote the smash product $\mr{X} \wedge \SS_{\mr{K}(2)}$\footnote{Throughout this section it is important to distinguish between $\widehat{\mr{X}}$ and the $\mr{K}(2)$-localization of $\mr{X}$. These are not equivalent in general  because the $\mr{K}(2)$-localization functor is not smashing.}. 
\end{notn}

 The work in \cite{Pha23} shows that the $\mr{K}(2)$-local Hurewicz map of $\A_1$ 
\[ 
\begin{tikzcd}
{\sf h}_{\mr{TMF}}: \pi_* \widehat{\A}_1 \rar[two heads] & \mr{TMF}_*\A_1
\end{tikzcd}
\]
is a surjection. Since  $\Delta^8$ acts faithfully on $\tmf_*\A_1$, $\tmf_*\Y$, $\tmf_*\M$, and $\tmf_*$ (see \Cref{rmk:injection}), the natural map
\[ 
\begin{tikzcd}
\ell_\ast : \tmf_*\mr{X} \rar & \mr{TMF}_* \mr{X}
\end{tikzcd}
\]
is an injection  when $\mr{X}$ is $\A_1$, $\Y$, $\M$, or $\mathbb{S}$. This injectivity allows us to establish that 
\begin{itemize}
\item[-] the image of $\ell_\ast(a) \in \mr{TMF}_*\A_1$ under the composite map 
\[ 
\begin{tikzcd}
p_2 \circ p_3: \mr{TMF}_* \A_1 \rar & \mr{TMF}_* \M
\end{tikzcd}
\] is nonzero if and only if the original element $p_2(p_3(a)) \in \tmf_*\M$ is nonzero, and
\item[-] the image of $\ell_\ast(a) \in \mr{TMF}_*\A_1$ under the map 
\[ 
\begin{tikzcd}
p_1 \circ p_2 \circ p_3: \mr{TMF}_* \A_1 \rar &  \mr{TMF}_{*-6}
\end{tikzcd}
\] is zero if and only if the element $p_1(p_2(p_3(a))) \in \tmf_{*-6}$ is zero.
\end{itemize}
Therefore the proof of \Cref{MT:K2} can follow the exact same arguments as the proof of \Cref{MT:Families} (as detailed in \Cref{sec:proofinffamily}), provided that the elements in the set
\[ \{ 8\Delta, 4\Delta^2, 8 \Delta^3, (\eta\Delta)^3, 2\Delta^4, 8 \Delta^5, 8\Delta^7\}\] 
are not in the image of \eqref{K2Hur}, the $\mr{K}(2)$-local Hurewicz image of $\mr{TMF}_*$ (also see \Cref{rmk:HurP}).
 \begin{thm}\label{MT:K2}
 \emph{ All elements listed in \Cref{MT:Families}  have nonzero images in the  $\mr{K}(2)$-local stable stems. }
\end{thm}
\begin{proof} Since the elements  
\[ 8\Delta, 4\Delta^2, 8 \Delta^3,  2\Delta^4, 8 \Delta^5, 8\Delta^7 \]  
are of infinite order, and $\pi_*\widehat{\SS}$  is a finite group in degrees $24, 48, 72, 96, 120$, and $168$ (see \cite[Theorem A]{BSSW24}), they cannot be in the $\mr{K}(2)$-local Hurewicz image. The remaining element $(\eta\Delta)^3$ is treated separately 
in the subsequent \Cref{lem:K2etaDelta3}. Once its exclusion from the Hurewicz image is established, the proof of  \Cref{MT:Families} goes through \emph{mutatis mutandis} to yield the result. 
\end{proof}
\begin{lemma} \label{lem:K2etaDelta3} The element $(\eta \Delta)^3$ is not in the image of the $\mr{K}(2)$-local Hurewicz map of $\mr{TMF}$
\begin{equation} \label{K2HI}
\begin{tikzcd}
{\sf h}_{\mr{TMF}} : \pi_* \widehat{\SS}\rar &  \mr{TMF}_*.
\end{tikzcd}
\end{equation}
\end{lemma}
\begin{proof}  From \cite[Corollary 3]{Lau04} (also see \cite[pg. 193]{DFHH14}), we  get 
\[
v_1^{-1} \mr{TMF} \simeq \mr{TMF}_{\mr{K}(1)} \simeq  \left(\mr{KO}\llbracket \mathfrak{j}^{\pm 1}\rrbracket \right)^{\sma}_2 \simeq \mr{KO}^{\sma}_2 (\!( \mathfrak{j}^{-1})\!),
\]
where $\mr{KO}^{\sma}_2$ denotes the $2$-adic completion of $\mr{KO}$. Therefore the composition of the $\mr{K}(2)$-local Hurewicz map with the $\mr{K}(1)$-localization map factors through $\mr{KO}^{\sma}_2$:
\[ 
\begin{tikzcd}
& \pi_*\mr{KO}_2^{\sma}  \ar[dr, hook] \\
\pi_* \widehat{\SS} \rar["{\sf h}_{\mr{TMF}}"'] \ar[ru] & \pi_* \mr{TMF} \rar & v_1^{-1} \mr{TMF}_* \cong \pi_*\mr{KO}^{\sma}_2 (\!( \mathfrak{j}^{-1})\!)
\end{tikzcd}
\]
Now we  implement an argument very similar to that of \cite[Thm. 6.1]{BMQ23}.

More precisely, we observe that $(\eta \Delta)^3$ lifts to an element in $\mr{TMF}_{\ast } \mr{M}({\infty}),$ where 
$ \mr{M}({\infty}) :=\underset{i \to \infty}{\mr{hocolim}} \ \mr{M}(i)$ (see \Cref{notn:M}), 
  whose image after inverting $c_4$ is 
   \[  \overline{v_1^{38} } \mathfrak{j}^{-3} \in v_1^{-1}\mr{TMF}_{\ast} \mr{M}({\infty}) \] 
as in \cite[$\mathsection$6]{BMQ23}. If $(\eta \Delta)^3$ is in the image of the Hurewicz map \eqref{K2HI}, then $\overline{v_1^{38} } \mathfrak{j}^{-3}$ must also be in the image of  $\ell \circ {\sf h}_{\mr{TMF}}$ in the diagram:
  \[ 
\begin{tikzcd}
& \mr{KO}_*^{\sma} \mr{M}({\infty}) \ar[dr, hook, "\mathfrak{j} =0"] \\
\pi_* \widehat{\mr{M}({\infty})} \rar["{\sf h}_{\mr{TMF}}"'] \ar[ru] &  \mr{TMF}_* \mr{M}({\infty})\rar["\ell"'] & v_1^{-1} \mr{TMF}_* \mr{M}({\infty})
\end{tikzcd}
\]
However, this contradicts the fact that the composite $\ell \circ {\sf h}_{\mr{TMF}}$ factors through $\mr{KO}_*^{\sma} \mr{M}(\infty)$, because $\overline{v_1^{38} } \mathfrak{j}^{-3}$ has a negative power of $\mathfrak{j}$, meaning it lies outside the submodule $\mr{KO}_*^{\sma} \mr{M}({\infty})$ of $v_1^{-1} \mr{TMF}_* \mr{M}({\infty})$. 
\end{proof}

 \bibliographystyle{plain}
\bibliography{InfiniteFamily}

\begin{thebibliography}{10}

\bibitem{Ada66}
John~Frank Adams.
\newblock On the groups {J}({X}) - {I}{V}.
\newblock {\em Topology}, 5(1):21--71, 1966.

\bibitem{BSSW24}
Tobias Barthel, Tomer~M. Schlank, Nathaniel Stapleton, and Jared Weinstein.
\newblock On the rationalization of the k(n)-local sphere.
\newblock available at https://arxiv.org/abs/2402.00960, 2024.

\bibitem{Bau08}
Tilman Bauer.
\newblock Computation of the homotopy of the spectrum tmf.
\newblock In {\em Proceedings of the conference on groups, homotopy and
  configuration spaces, University of Tokyo, Japan, July 5--11, 2005 in honor
  of the 60th birthday of Fred Cohen}, pages 11--40. Coventry: Geometry \&
  Topology Publications, 2008.

\bibitem{Beau15}
Agn\`es Beaudry.
\newblock The algebraic duality resolution at {$p=2$}.
\newblock {\em Algebr. Geom. Topol.}, 15(6):3653--3705, 2015.

\bibitem{BBBCX21}
Agn\`es Beaudry, Mark Behrens, Prasit Bhattacharya, Dominic Culver, and Zhouli
  Xu.
\newblock The telescope conjecture at height 2 and the tmf resolution.
\newblock {\em J. Topol.}, 14(4):1243--1320, 2021.

\bibitem{BBGHPS24}
Agn\`es Beaudry, Irina Bobkova, Paul~G. Goerss, Hans-Werner Henn, Viet-Cuong
  Pham, and Vesna Stojanoska.
\newblock Cohomology of the {M}orava stabilizer group through the duality
  resolution at {$n = p = 2$}.
\newblock {\em Trans. Amer. Math. Soc.}, 377(3):1761--1805, 2024.

\bibitem{BBH}
Agn\`es Beaudry, Irina Bobkova, and Hans-Werner Henn.
\newblock The duality resolution at {$n=p=2$}.
\newblock {\em Math. Z.}, 310(3):Paper No. 50, 38, 2025.

\bibitem{BBPX22}
Agn{\`e}s Beaudry, Irina Bobkova, Viet-Cuong Pham, and Zhouli Xu.
\newblock The topological modular forms of {{\(\mathbb{R}P^2\)}} and
  {{\(\mathbb{R}P^2 \wedge \mathbb{C}P^2\)}}.
\newblock {\em J. Topol.}, 15(4):1864--1926, 2022.

\bibitem{BGH}
Agn\`es Beaudry, Paul~G. Goerss, and Hans-Werner Henn.
\newblock Chromatic splitting for the {$K(2)$}-local sphere at {$p = 2$}.
\newblock {\em Geom. Topol.}, 26(1):377--476, 2022.

\bibitem{BHHM20}
M.~Behrens, M.~Hill, M.~J. Hopkins, and M.~Mahowald.
\newblock Detecting exotic spheres in low dimensions using coker {{\(J\)}}.
\newblock {\em J. Lond. Math. Soc., II. Ser.}, 101(3):1173--1218, 2020.

\bibitem{Beh}
Mark Behrens.
\newblock A modular description of the {$K(2)$}-local sphere at the prime 3.
\newblock {\em Topology}, 45(2):343--402, 2006.

\bibitem{Behrens5}
Mark Behrens.
\newblock The homotopy groups of {$S_{E(2)}$} at {$p\geq 5$} revisited.
\newblock {\em Adv. Math.}, 230(2):458--492, 2012.

\bibitem{Btmf}
Mark Behrens.
\newblock Topological modular and automorphic forms.
\newblock In {\em Handbook of homotopy theory}, CRC Press/Chapman Hall Handb.
  Math. Ser., pages 221--261. CRC Press, Boca Raton, FL, [2020] \copyright
  2020.

\bibitem{BBC23}
Mark Behrens, Prasit Bhattacharya, and Dominic Culver.
\newblock The structure of the $v_2$-local algebraic tmf resolution.
\newblock available at https://arxiv.org/abs/2301.11230, 2023.

\bibitem{BMQ23}
Mark Behrens, Mark Mahowald, and J.D. Quigley.
\newblock The 2-primary {H}urewicz image of $\mathit{tmf}$.
\newblock {\em Geometry \& Topology}, 27(7):2763--2831, 2023.

\bibitem{BO16}
Mark Behrens and Kyle Ormsby.
\newblock On the homotopy of {$Q(3)$} and {$Q(5)$} at the prime 2.
\newblock {\em Algebr. Geom. Topol.}, 16(5):2459--2534, 2016.

\bibitem{BOSS19}
Mark Behrens, Kyle Ormsby, Nathaniel Stapleton, and Vesna Stojanoska.
\newblock On the ring of cooperations for 2-primary connective topological
  modular forms.
\newblock {\em Journal of Topology}, 12(2):577--657, 2019.

\bibitem{BP04}
Mark Behrens and Satya Pemmaraju.
\newblock On the existence of the self map {$v^9_2$} on the {S}mith-{T}oda
  complex {$V(1)$} at the prime 3.
\newblock In {\em Homotopy theory: relations with algebraic geometry, group
  cohomology, and algebraic {$K$}-theory}, volume 346 of {\em Contemp. Math.},
  pages 9--49. Amer. Math. Soc., Providence, RI, 2004.

\bibitem{BBT21}
Prasit Bhattacharya, Irina Bobkova, and Brian Thomas.
\newblock The {$P^1_2$} {M}argolis homology of connective topological modular
  forms.
\newblock {\em Homology Homotopy Appl.}, 23(2):379--402, 2021.

\bibitem{BElocal20}
Prasit Bhattacharya and Philip Egger.
\newblock Towards the {$K(2)$}-local homotopy groups of {$Z$}.
\newblock {\em Algebr. Geom. Topol.}, 20(3):1235--1277, 2020.

\bibitem{BEM17}
Prasit Bhattacharya, Philip Egger, and Mark Mahowald.
\newblock On the periodic {{\(v_2\)}}-self-map of {{\(A_1\)}}.
\newblock {\em Algebr. Geom. Topol.}, 17(2):657--692, 2017.

\bibitem{Bobkova20}
Irina Bobkova.
\newblock Spanier-{W}hitehead duality in the {$K(2)$}-local category at
  {$p=2$}.
\newblock {\em Proc. Amer. Math. Soc.}, 148(12):5421--5436, 2020.

\bibitem{BG18}
Irina Bobkova and Paul~G. Goerss.
\newblock Topological resolutions in {$K(2)$}-local homotopy theory at the
  prime 2.
\newblock {\em J. Topol.}, 11(4):918--957, 2018.

\bibitem{BousfieldLoc}
A.~K. Bousfield.
\newblock The localization of spectra with respect to homology.
\newblock {\em Topology}, 18(4):257--281, 1979.

\bibitem{Bro69}
William Browder.
\newblock The {K}ervaire invariant of framed manifolds and its generalization.
\newblock {\em Ann. of Math. (2)}, 90:157--186, 1969.

\bibitem{BR21}
Robert~R. Bruner and John Rognes.
\newblock {\em The {Adams} spectral sequence for topological modular forms},
  volume 253 of {\em Math. Surv. Monogr.}
\newblock Providence, RI: American Mathematical Society (AMS), 2021.

\bibitem{BHLS}
Robert Burklund, Jeremy Hahn, Ishan Levy, and Tomar~M. Schlank.
\newblock K-theoretic counterexamples to ravenel's telescope conjecture.
\newblock available at https://arxiv.org/abs/2310.17459, 2023.

\bibitem{DMv1v2}
Donald~M. Davis and Mark Mahowald.
\newblock {$v\sb{1}$}- and {$v\sb{2}$}-periodicity in stable homotopy theory.
\newblock {\em Amer. J. Math.}, 103(4):615--659, 1981.

\bibitem{DHS88}
Ethan~S. Devinatz, Michael~J. Hopkins, and Jeffrey~H. Smith.
\newblock Nilpotence and stable homotopy theory. {I}.
\newblock {\em Ann. of Math. (2)}, 128(2):207--241, 1988.

\bibitem{DFHH14}
Christopher~L Douglas, John Francis, Andr{\'e}~G Henriques, and Michael~A Hill.
\newblock {\em Topological modular forms}, volume 201.
\newblock American Mathematical Soc., 2014.

\bibitem{GHMV1}
P.~Goerss, H.-W. Henn, and M.~Mahowald.
\newblock The homotopy of {$L_2V(1)$} for the prime 3.
\newblock 215:125--151, 2004.

\bibitem{GHMR}
P.~Goerss, H.-W. Henn, M.~Mahowald, and C.~Rezk.
\newblock A resolution of the {$K(2)$}-local sphere at the prime 3.
\newblock {\em Ann. of Math. (2)}, 162(2):777--822, 2005.

\bibitem{HHR16}
Michael~A Hill, Michael~J Hopkins, and Douglas~C Ravenel.
\newblock On the nonexistence of elements of {K}ervaire invariant one.
\newblock {\em Annals of Mathematics}, 184(1):1--262, 2016.

\bibitem{HS98}
Michael~J. Hopkins and Jeffrey~H. Smith.
\newblock Nilpotence and stable homotopy theory. {II}.
\newblock {\em Ann. of Math. (2)}, 148(1):1--49, 1998.

\bibitem{HH67}
Wu-chung Hsiang and Wu-yi Hsiang.
\newblock On compact subgroups of the diffeomorphism groups of {K}ervaire
  spheres.
\newblock {\em Ann. of Math. (2)}, 85:359--369, 1967.

\bibitem{Isa19}
Daniel~C. Isaksen.
\newblock Stable stems.
\newblock {\em Mem. Amer. Math. Soc.}, 262(1269):viii+159, 2019.

\bibitem{IWX23}
Daniel~C. Isaksen, Guozhen Wang, and Zhouli Xu.
\newblock Stable homotopy groups of spheres: from dimension 0 to 90.
\newblock {\em Publ. Math. Inst. Hautes \'{E}tudes Sci.}, 137:107--243, 2023.

\bibitem{KM63}
Michel~A. Kervaire and John~W. Milnor.
\newblock Groups of homotopy spheres. {I}.
\newblock {\em Ann. Math. (2)}, 77:504--537, 1963.

\bibitem{KM93}
Stanley~O. Kochman and Mark~E. Mahowald.
\newblock On the computation of stable stems.
\newblock In {\em The \v{C}ech centennial ({B}oston, {MA}, 1993)}, volume 181
  of {\em Contemp. Math.}, pages 299--316. Amer. Math. Soc., Providence, RI,
  1995.

\bibitem{Kuhn}
Nicholas~J. Kuhn.
\newblock A guide to telescopic functors.
\newblock {\em Homology Homotopy Appl.}, 10(3):291--319, 2008.

\bibitem{Lau04}
Gerd Laures.
\newblock {$K(1)$}-local topological modular forms.
\newblock {\em Invent. Math.}, 157(2):371--403, 2004.

\bibitem{Ltmf}
Jacob Lurie.
\newblock A survey of elliptic cohomology.
\newblock In {\em Algebraic topology}, volume~4 of {\em Abel Symp.}, pages
  219--277. Springer, Berlin, 2009.

\bibitem{MT67}
Mark Mahowald and Martin Tangora.
\newblock Some differentials in the {A}dams spectral sequence.
\newblock {\em Topology}, 6:349--369, 1967.

\bibitem{MRW77}
Haynes~R Miller, Douglas~C Ravenel, and W~Stephen Wilson.
\newblock Periodic phenomena in the {A}dams-{N}ovikov spectral sequence.
\newblock {\em Annals of Mathematics}, 106(3):469--516, 1977.

\bibitem{Pha23}
Viet-Cuong Pham.
\newblock On the surjectivity of the tmf-{Hurewicz} image of {{\(A_1\)}}.
\newblock {\em Algebr. Geom. Topol.}, 23(1):217--241, 2023.

\bibitem{Rav86}
Douglas~C. Ravenel.
\newblock {\em Complex cobordism and stable homotopy groups of spheres}, volume
  121 of {\em Pure and Applied Mathematics}.
\newblock Academic Press, Inc., Orlando, FL, 1986.

\bibitem{Orange}
Douglas~C. Ravenel.
\newblock {\em Nilpotence and periodicity in stable homotopy theory}, volume
  128 of {\em Annals of Mathematics Studies}.
\newblock Princeton University Press, Princeton, NJ, 1992.
\newblock Appendix C by Jeff Smith.

\bibitem{Sch85}
Reinhard Schultz.
\newblock Transformation groups and exotic spheres.
\newblock Group actions on manifolds, {Proc}. {AMS}-{IMS}-{SIAM} {Joint}
  {Summer} {Res}. {Conf}., {Boulder}/{Colo}. 1983, {Contemp}. {Math}. 36,
  243-267 (1985)., 1985.

\bibitem{Ser53}
Jean-Pierre Serre.
\newblock Groupes d'homotopie et classes de groupes ab{\'e}liens.
\newblock {\em Ann. Math. (2)}, 58:258--294, 1953.

\bibitem{ShimomuraWang}
Katsumi Shimomura and Xiangjun Wang.
\newblock The homotopy groups {$\pi_*(L_2S^0)$} at the prime 3.
\newblock {\em Topology}, 41(6):1183--1198, 2002.

\bibitem{SY}
Katsumi Shimomura and Atsuko Yabe.
\newblock The homotopy groups {$\pi_*(L_2S^0)$}.
\newblock {\em Topology}, 34(2):261--289, 1995.

\bibitem{Smi77}
Larry Smith.
\newblock On realizing complex bordism modules. {IV}: {Applications} to the
  stable homotopy groups of spheres.
\newblock {\em Am. J. Math.}, 99:418--436, 1977.

\bibitem{Tod62}
Hirosi Toda.
\newblock {\em Composition methods in homotopy groups of spheres}, volume~49 of
  {\em Ann. Math. Stud.}
\newblock Princeton University Press, Princeton, NJ, 1962.

\bibitem{WX17}
Guozhen Wang and Zhouli Xu.
\newblock The triviality of the 61-stem in the stable homotopy groups of
  spheres.
\newblock {\em Annals of Mathematics}, 186(2):501--580, 2017.

\bibitem{Wra97}
David Wraith.
\newblock Exotic spheres with positive {Ricci} curvature.
\newblock {\em J. Differ. Geom.}, 45(3):638--649, 1997.

\end{thebibliography}

\end{document}